\documentclass[11pt]{article}
\usepackage{amssymb}
\usepackage{mathrsfs,pstricks}
\usepackage{amsfonts}
\usepackage{latexsym}
\usepackage{amsmath,amssymb}
\usepackage{amsthm}
\usepackage{amsfonts}
\usepackage{ifthen,calc}
\usepackage{color}
\usepackage{enumerate}
\usepackage{tikz}
\usepackage{epsfig}
\usepackage{graphicx}
\usepackage{float}
\usepackage{latexsym}
\usepackage{multicol}
\usepackage{subfigure}
\usepackage{geometry}
\usepackage{setspace}
\usepackage{CJK}
\usepackage{mathrsfs}
\usepackage{amsmath}
\usepackage{amssymb,mathrsfs,amsthm,amsfonts}
\usepackage{graphicx}
\usepackage[pdfstartview=FitH]{hyperref}
\usepackage[title]{appendix}
\usepackage{color}
\usepackage{tikz}
\usepackage{epsfig}
\usepackage{subfigure}

\providecommand{\U}[1]{\protect\rule{.1in}{.1in}}
\allowdisplaybreaks % ����ʵ�ֶ��й�ʽ�Զ���ҳ

\newtheorem{theorem}{Theorem}[section]

\newtheorem{proposition}[theorem]{Proposition}
\newtheorem{corollary}[theorem]{Corollary}
\newtheorem{definition}[theorem]{Definition}

\newtheorem{remark}[theorem]{Remark}
\newtheorem{lemma}[theorem]{Lemma}

\newtheorem{example}[theorem]{Example}

\def\eps{\varepsilon}
\def\]{{\Big]}}
\def\[{{\Big[}}

\def\bd{\begin{definition}}
\def\ed{\end{definition}}
\def\bp{\begin{proposition}}
\def\ep{\end{proposition}}
\def\bc{\begin{corollary}}
\def\ec{\end{corollary}}
\def\bx{\begin{Examples}}
\def\ex{\end{Examples}}

\def\cE{{\mathcal E}}

\def\cL{{\mathcal L}}

\def\mZ{{\mathbb Z}}

\def\ba{{\begin{align}}
\def\ea{\end{align}}}

\def\geq{\geqslant}
\def\leq{\leqslant}

\def\^{\widehat}

\def\E{\mathbb E}

\def\ba{\begin{aligned}}
\def\ea{\end{aligned}}
\def\be{\begin{equation}}
\def\ee{\end{equation}}
\def\ben{\begin{align*}}
\def\enn{\end{align*}}

\def\E{\mathcal{E}}

\newcommand{\mdef}[1]{\textit{{#1}}}
\newcommand{\vare}{\varepsilon}
\def\dsum{\displaystyle\sum}

\makeatletter

\newcommand{\Rmnum}[1]{\expandafter\@slowromancap\romannumeral #1@}
\makeatother

\numberwithin{equation}{section}

\title{\bf{Euler-genus distributions of cubic caterpillar-Halin graphs}}

{\author{ Jinlian Zhang$^\dag$  \\
{\em\small department of Mathematics, Hunan University, 410082 Changsha, China}\\
{\ Xuhui Peng$^\spadesuit$} \\
{\em\small department  of  statistics,  Hunan Normal University, 410081 Changsha, China }\\
 }

 \date{}

\begin{document}

\maketitle
\let\thefootnote\relax\footnotetext{$^\dag$ Email: jinlian916@hnu.edu.cn.;\quad \quad $^\spadesuit$ Email:  xhpeng@hunnu.edu.cn}

\begin{abstract}
\noindent
 Gross \cite{Gro11} derived an $O(n^2)$-time algorithm to  calculate  the genus distribution of a given cubic Halin graph.
 In this  paper, with the help of overlap matrix,  we get
   a   recurrence   relation  for the     Euler-genus polynomial of  cubic caterpillar-Halin graphs.
   Explicit formulas  for the
   embeddings of   cubic caterpillar-Halin graph  into a surface with Euler-genus 0, 1 and 2 are  also obtained.

\vspace{1mm}
\end{abstract}
\section{\Large Introduction}

%\subsection{Background}  %%%%%%%%%%%%%%%%%%%%%%%

%A \mdef{surface} $S$ is a compact $2$-dimensional manifold without boundary.
There are two kinds of  surfaces  in topology. One is the \mdef{orientable surfaces} $O_j$, which has  $j$ handles $(j\ge0)$ and the other is the \mdef{non-orientable surfaces} $N_k$, which has  $k$ crosscaps $(k>0)$.
%The graph embeddings  considered  here are \mdef{cellular embeddings}.
The research of  graph embeddings on non-orientable surfaces was  began   by  Chen, Gross and Rieper \cite{CGR94}.
The authors in  \cite{CG18} introduced   a combinatorial way   to calculate   the Euler-genus polynomial of a graph.
Recently,  an $O(n^2)$-time algorithm to  calculate  the genus distribution of any  cubic Halin graph $H_T$ was derived  by  Gross \cite{Gro11}.
 For  linear graph families, %(Mohar \cite{Moh15} called it Path-like graph families),
  Stahl \cite{Sta91} calculated their  genus distributions.
 Enami \cite{Ena18} made researches on the  inequivalent embeddings for 3-connected 3-regular planar graph. %on the torus. %  Cubic Halin graphs are important graph families with tree-width $3$.
  However, there  still  no results on  recursive formulas for  Euler-genus distributions of any cubic Halin graph.
 In this paper, we make researches  on cubic caterpillar-Halin graphs and obtain    their   recurrence   relation for  Euler-genus polynomial.
 %thus, one may hope to find recursive formulas or explicit formulas for embedding distributions of any cubic Halin graph.

\subsection{Euler-genus polynomial}

The Euler-genus $\gamma^E$ of a surface $S$ is
\begin{eqnarray*}
  \gamma^E=
  \left\{
  \begin{split}
   &  2j, \quad \text{if }  S=O_j,
    \\
    &  k,\quad  \text{ if }  S=N_k.
  \end{split}
  \right.
\end{eqnarray*}
In this paper, $S_i$ denotes  a  surface $S$ with Euler-genus $i$.

 A \emph{pure rotation system} $\rho$ of graph  $G=(V(G),E(G))$ is the set $\{\rho_v: v\in V(G)\},$
where   $\rho_v$  is  a cyclic permutation of edges incident to $v$.
Assume  $\lambda$ is a mapping $E(G)\rightarrow\{0,1\}$.
We say   edge  $e$ is  \mdef{untwisted}  if $\lambda(e)=0$,  and we say it   is  \mdef{twisted} if $\lambda(e)=1$.
 This   pair $(\rho,\lambda)$ is called a
\mdef{general rotation system} for graph  $G$.
% It is well-known that every  embedding of a graph $G$ can be described by a general rotation system $(\rho,\lambda)$.
Let $T$ be a spanning tree  of graph $G$.
If $\lambda(e) = 0,\forall e\in E(T),$ this general rotation system $(\rho,\lambda)$  is called
 a $T$-rotation system.
For any two embeddings of $G$, if their $T$-rotation systems are combinatorially
equivalent, we consider them  to be \textit{equivalent}.
Let  $\vare_G(i)$ be  the number of \textit{equivalent embeddings} of $G$ into the surface $S_i,$ for $i\geq 0.$
 The \textit{Euler-genus distribution}  of graph $G$ is the sequence
 $
 \vare_G(0),\vare_G(1),\cdots,\vare_G(i),\cdots
 $
 and the \textit{Euler-genus polynomial} of graph $G$ is $\E_G(z)=\dsum_{i=0}^\infty \vare_G(i)z^i.$

\subsection{Cubic  caterpillar-Halin graphs}

A \textit{cubic Halin graph} $H_T=T\cup C$, where  $T$ is a plane tree with 3 degree  interior vertex
and  $C$ is a cycle which  connects  the leaves
(vertices of degree 1) of $T$ so that $C$ is the boundary of an  \mdef{infinite face.}
Except this \mdef{infinite face}, the  other faces  are  called  as \mdef{interior faces}.
This  plane tree $T$   is called the characteristic tree of  graph $H_T$.
%\textcolor[rgb]{1.00,0.00,0.00}{
% An tree,  if it   has no cycles and there is a path in the graph such that  every node is either on this path or a neighbor of a node on the path,
% is called a caterpillar.  This path is called the spine of the caterpillar.}
 A tree $T$,   whose removal of leaves results in a path $P$,  is called a \textit{caterpillar}
 and this path $P$ is  called the \mdef{spine} of the  caterpillar $T$.
%This path is called  the spine of the caterpillar.
 When  the characteristic tree of $H_T$ is a caterpillar,  we call such  graph $H_T$  \textit{cubic  caterpillar-Halin graph}.

Let $T$ be a caterpillar with spine  $P : v_1, v_2, \ldots, v_
 \ell$.   We  consider the   cubic  caterpillar-Halin graph  $H_{T}.$
 In the spine  $P,$  each vertex $v_i$ is adjacent to a leaf $u_i$ for $1 \leq i \leq \ell $ and  $v_1$ (resp. $v_\ell$ ) adjacent to one more leaf $u_0 = v_0$ (resp. $u_{\ell+1} = v_{\ell+1}$).
 We put  the path $v_0Pv_{\ell+1}$ horizontally in the middle, and then pend these  edges  $v_iu_i,$ $1 \leq i \leq \ell$ by either up or down edges vertically to $P.$ The graph in  Figure \ref{tu1} is  a 	concrete  example  for  cubic  caterpillar-Halin graph.
\begin{figure}[h]
  \centering
  \includegraphics[width=0.82\textwidth]{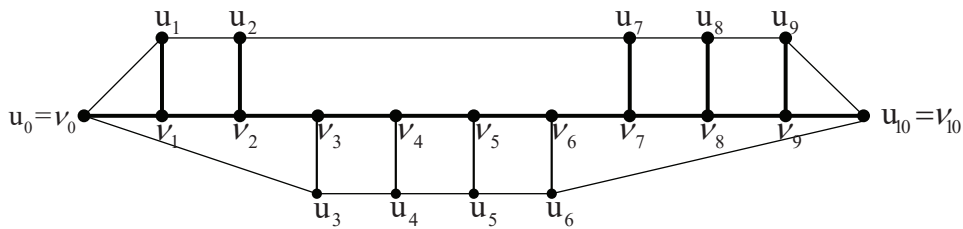}
  \caption{A cubic caterpillar-Halin graph: $H_{2,3,3}$}
  \label{tu1}
\end{figure}

Counting   from the leftmost to the rightmost on $P,$ if  the numbers of maximum consecutive up or down edges
of  the caterpillar-Halin graph $H_T$   are  $m_1,m_2+1,\cdots,m_{k-1}+1,m_k$,  we denote   this graph $H_T$ by  $H_{m_1,m_2,\ldots,m_k}$.
 In order to have a more simple form of the overlap matrix for
 graph $H_{m_1,m_2,\ldots,m_k}$, here  we  don't use   $m_1,m_2,\cdots,m_{k-1},m_k$ to denote
the numbers  of maximum consecutive up or down edges.
For example, the graph in Figure \ref{tu1} is $H_{2,3,3}$.
Figure \ref{1-11}  demonstrates the general case $H_{m_1,m_2,\ldots,m_k}.$
\begin{figure}[h]
  \centering
  \includegraphics[width=1.06\textwidth]{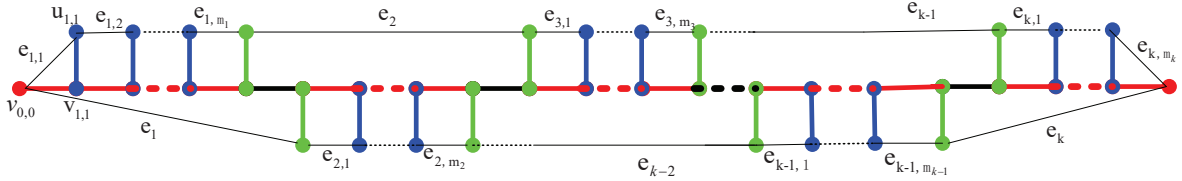}
  \caption{Graph $H_{m_1,\cdots,m_k}$}
  \label{1-11}
\end{figure}

For convenience,  our choice of a spanning tree $T$ for
$H_{m_1,\cdots,m_k}$ is indicated by    \textit{{thicker lines}}.
Then, as shown in Figure \ref{1-11},  the co-tree edges are
$$e_{1},\cdots e_{k},e_{1,1},\cdots,e_{1,m_{1}},e_{2,1},
\cdots,e_{2,m_{2}},\cdots,e_{k-1,1},\cdots,e_{k-1,m_{k-1}}
e_{k,1},\cdots,e_{k,m_k}.$$
If we delete the edges $\{e_1,\cdots,e_k\}$ in graph $H_{m_1,\cdots,m_k}$,
then  $m_1,\cdots,m_k$  are   the numbers of maximum consecutive co-tree  edges actually.

Now, we consider a special case of  $H_{m_1,m_2,\ldots,m_k}.$   When $k=1$ and $m_1=m\geq 1 $, this graph  $H_{m}$ is usually called   a   Ringel ladder graph    denoted by $R_{m}.$
The explicit formula of the embedding distributions for $R_m$ was obtained by  Chen et al. \cite{COZ11}.
% Figure \ref{fig:three}  shows the graphs $R_1,R_2,R_3.$
 Figure \ref{fig:G4} demonstrates  the  graphs  $R_4$ and  $H_{2,2}$.

In this paper, for any  positive integer $m_1,$  the graph $H_{m_1,0}$ is defined to be  $H_{m_1}$ or $R_{m_1}$.  For example, the graph $R_4$ in Figure  \ref{fig:G4} is  also $H_{4,0}$.
\begin{figure}[h]
  \centering
  \includegraphics[width=0.86\textwidth]{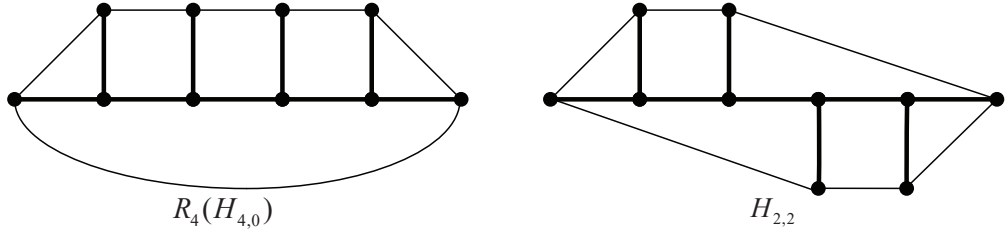}
  \caption{Graphs $R_4$ and $H_{2,2}$}
  \label{fig:G4}
\end{figure}
For any  $k\geq 2$  and  positive integers $m_1,\cdots,m_k$,
the graph  $H_{m_1,\cdots,m_k,0}$ is  defined to be the following graph in Figure \ref{fig-p1}.
\begin{figure}[h]
  \centering
  \includegraphics[width=1.08\textwidth]{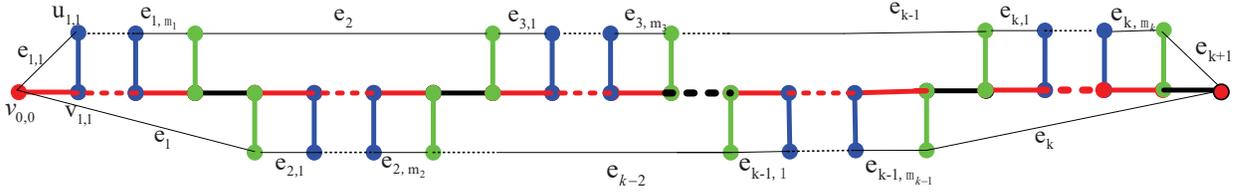}
  \caption{Graph $H_{m_1,\cdots,m_k,0}$}
  \label{fig-p1}
\end{figure}
Actually, the graph $H_{m_1,\cdots,m_k,0}$ can be  interpreted as the graph $H_{m_1,\cdots,m_{k},m_{k+1}}$ for the special case  of $m_{k+1}=0.$

\newpage

By the definitions above, we  observe the following propositions.

\begin{proposition}\label{2-11}
For  any  positive  integers $m_1,\cdots,m_k$,   $H_{m_1,m_2,\ldots,m_k} \cong H_{m_k,m_{k-1},\ldots,m_1}$
\end{proposition}

\begin{proposition}
\label{x-13}
 For any $k\geq 2$  and   positive  integers $m_1,\cdots,m_k$, we have $H_{m_1 ,m_2 ,\cdots,m_{k-1}, m_k ,0}
\cong H_{m_1 ,m_2,\cdots,m_{k-1}, m_k+1}$ and $H_{m_1,0}
\cong R_{m_1}.$
\end{proposition}

\begin{proposition}\label{2-12}
For any  $k\geq 2$ and positive   integers $m_1,\cdots,m_k$,  we have $H_{m_1,m_2,\ldots,m_k,1} \cong H_{m_1,m_2,\ldots,m_{k-1},m_k+2}$
and  $H_{m_1,1}
\cong R_{m_1+1}.$
\end{proposition}

\begin{figure}[h]
  \centering
  \includegraphics[width=0.75\textwidth]{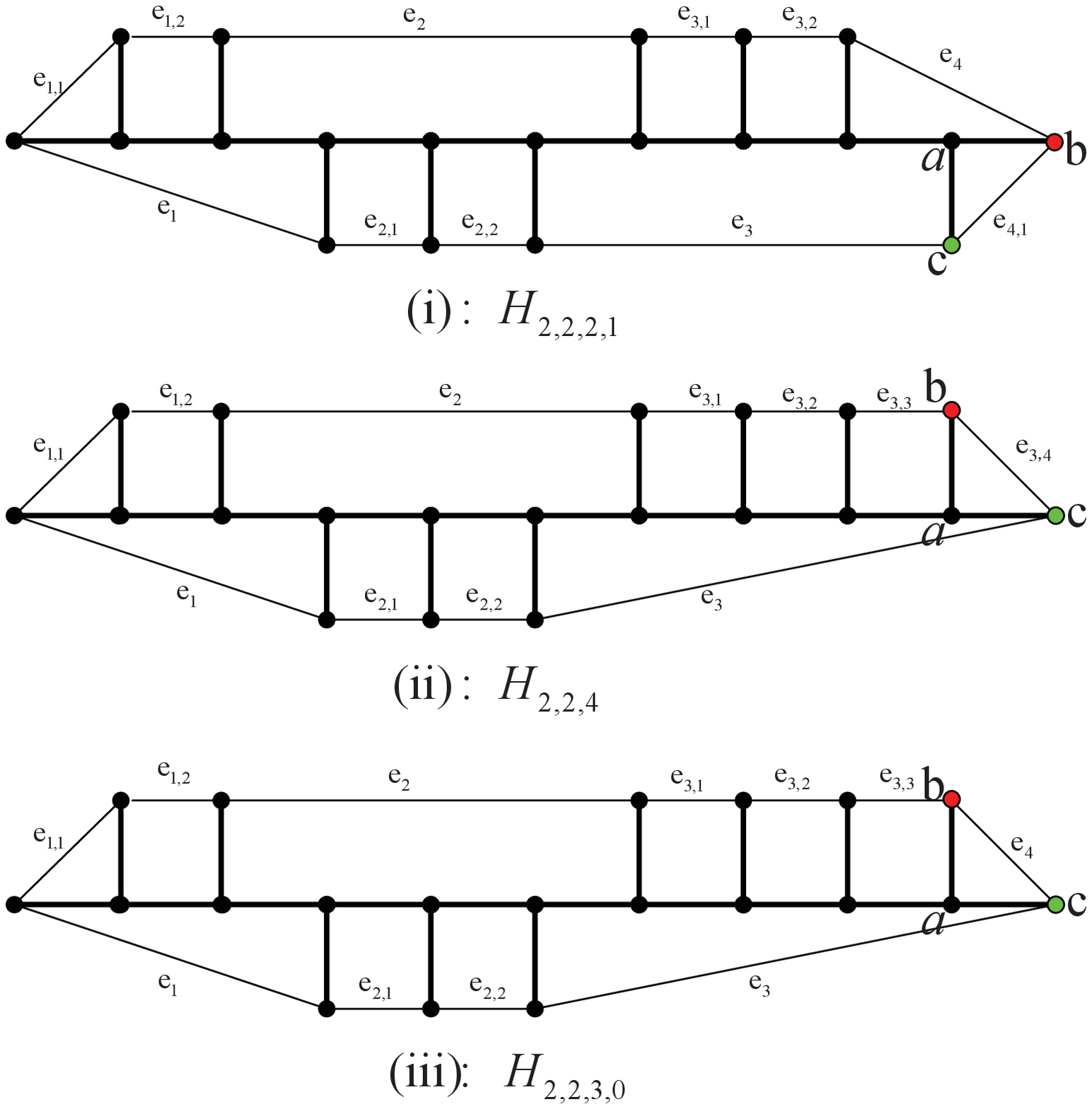}
  \caption{Graphs   $H_{2,2,2,1}$, $H_{2,2,4}$ and $H_{2,2,3,0}$}
  \label{fig:H10}
\end{figure}

The  graphs  $H_{2,2,4}$ and  $H_{2,2,3,0}$ in Figure  \ref{fig:H10}
demonstrate Proposition \ref{x-13}.
The   graphes $H_{2,2,2,1}$ and $H_{2,2,4}$ in Figure  \ref{fig:H10}
 demonstrate  the correction of Proposition \ref{2-12}.
Actually, all the three graphs in Figure  \ref{fig:H10} are the same.% Since there are different    ways to label the co-tree edges, we may use different   notations to denote the same graph.

\subsection{ Mohar's theorem} %%%%%%%%%%%%%%%%%%%%%%%%
%\subsection{Gustin's coloring for cubic Halin graphs}%%%%%%%%%%%%%%%
Let $G$ be a graph with  a spanning tree $T$,
$(\rho, \lambda)$ be a $T$-rotation system and   $\beta(G)$ be the \mdef{Betti number} of $G$.
We denote   the co-tree edges \hbox{of $T$} by  $e_1,e_2,\ldots,e_{\beta(G)}$. Then, the  \mdef{overlap matrix} of $(\rho, \lambda)$ is a  matrix $M=[m_{ij}]_{i,j=1}^{\beta(G)}$ over $\mZ_2=\{0,1\}$ with $m_{ij}$ given by
$$m_{ij}~=~
\begin{cases}
    1, &\text{if $i=j$ and $e_i$ is twisted;} \\
    1, &\text{if $i\neq j$ and the restriction of the underlying pure rotation system} \\[-2pt]
    &   \text{ to the subgraph $T+e_i+e_j$ is nonplanar;} \\
    0, &\text{otherwise.}
\end{cases}$$
For $i\neq j,$
when  $m_{ij}=1,$  the  edges $e_i$ and $e_j$ are called  \mdef{overlap}.  Mohar \cite{Moh89} proved
 that,   independent of the choice of  spanning tree,   the rank of matrix  $M$ is the same as   the Euler-genus of the corresponding embedding surface.
% It is .
%
% the following theorem.
%\begin{theorem} [Mohar \cite{Moh89}]  \label{thm:Moh89}
%Let $(\rho, \lambda)$ be a general rotation system for a graph, and let $M$ be the overlap matrix with respect to any spanning tree.
%\end{theorem}
%The following proposition is obvious.
%\begin{proposition}\label{zhang1}
%Let $T$ be a spanning tree of graph $G$, $e_i$ and $e_j$ are the two co-tree edges of $G$, then $e_i$ and $e_j$  overlap  if there is at least a common vertex in the base circle of $T$.
%\end{proposition}

Let  $\eps_{m_1,\cdots,m_k}(j)$ be  the number of \textit{equivalent embeddings} of graph  $H_{m_1,\cdots,m_k}$ into  surface $S_j,$ for $j\geq 0$ and $\cE_{m_1,\cdots,m_k}(z)=\dsum_{j=0}^\infty \eps_{m_1,\cdots,m_k}(j)z^j$ be the  \textit{Euler-genus polynomial} of $H_{m_1,\cdots,m_k}.$
We organize this paper as follows: The  overlap  matrices  for graph  $H_{m_1,\cdots,m_k}$ are given   in Section 2.
In Section 3,
 we obtain   a  recurrence   relation for   $\E_{m_1,\cdots,m_k}(z)$ in Theorem  \ref{5-1}.
 In Section 4,  we derive   a recurrence relation   for  $ \E^1(t_1,t_2,\cdots,t_k,z),$
 where
$$
 \E^1(t_1,t_2,\cdots,t_k,z)=\sum_{m_1,m_2,\cdots,m_k\geq 1} \E_{m_1,m_2,\cdots,m_k}(z)t_1^{m_1}t_2^{m_2}\cdots t_k^{m_k}.
$$
 In Section 5, some examples are demonstrated  to   illustrate our results.

\section{Overlap  matrices for cubic  caterpillar-Halin graphs}
\label{zhazha2}
For  a vertex $v$ with degree three  in  graph $G,$  it  has  two possibilities. We call them   \mdef{clockwise} and \mdef{counterclockwise}.
We color the vertex $v$ according to its rotation. If the   rotation  is \mdef{clockwise},  We color the vertex $v$ \mdef{black}; otherwise,   we color it \mdef{white}.
Therefore,   for any  3-regular graph,
 we associate their  rotation system  with a coloring  which is  called  \mdef{Gustin coloring}.
  In a Gustin coloring, an edge with the same color at both ends  is called \mdef{matched}.
The other edges  are called \mdef{unmatched}.

The following propositions  are useful in the calculation of the  overlap matrix for  $H_{m_1,\cdots,m_k}$.  Applying   the Heffter-Edmonds face-tracing algorithm,  we can obtain their proof.

\begin{proposition}\label{jin1}

 Two co-tree edges  overlap if and only if the common edge of
the  corresponding two interior faces   is unmatched.
\end{proposition}

\begin{proposition}\label{jin5}

 The  edges $e_i$ and
 $e_{i-1,m_{i-1}}$$(e_i$ and $e_{i+1,1})$  overlap,  for $i\in \{2,\cdots,k\}$$(i\in \{1,\cdots,k-1\})$, if and only if the corresponding  green edge in Figure \ref{1-11}  is unmatched.
\end{proposition}

\begin{proposition}\label{jin2}
 The  edges $e_i$ and $e_{i+1}$ overlap,  for $i\in \{1,\cdots,k-1\}$,
 if and only if the corresponding thicker black  edge in Figure \ref{1-11}   is unmatched.
\end{proposition}

\begin{proposition}\label{jin3}
The edges $e_i$ and $e_{i,\ell}$ overlap,  for $i\in \{1,\cdots,k\}$,$1\leq \ell\leq m_i$ if and only if the corresponding  red edge in Figure \ref{1-11}   is unmatched.
\end{proposition}

\begin{proposition}\label{jin4}

 The  edges $e_{i,\ell}$ and
 $e_{i,\ell +1}$ overlap,  for $i\in \{1,\cdots,k\}$,$1\leq \ell\leq m_i-1 $ if and only if the correspondence blue edge in Figure \ref{1-11}   is unmatched.
\end{proposition}

For  any $n\geq m$  and  $a=(a_1,\cdots,a_n)\in (\mZ_2)^{n},b=(b_1,\cdots,b_{n-1}) \in (\mZ_2)^{n-1} $,
we define $\mathcal{L}_m^{a,b}$ as a matrix of this  form:
\begin{eqnarray}
\label{x-2}
  \begin{bmatrix}
     a_1 & b_1  &  &\textbf{0 }
     \\
     b_1 & a_2  & b_2
  \\
    &  b_2 & a_3  & \ddots
     \\
   \textbf{0} && \ddots&\ddots   & b_{m-1}
     \\    & &   &    b_{m-1} &a_{m}
    \end{bmatrix}.
\end{eqnarray}
Let
$
  \cL_m=\{\mathcal{L}_m^{a,b}: a_i\in \mZ_2, i=1,\cdots,m;  b_j\in \mZ_2, j=1,\cdots,m-1\}
$
and  $L_m (z)=\sum_{j=0}^{m}L_{m}[j]z^j,$
where $L_{m}[j]$ is the   number of matrices of the form (\ref{x-2}) for  which this matrix  has rank $j$.

Let
\begin{eqnarray}
\label{1-2}
\begin{split}
\textbf{x}_0& =  (x_1,\cdots,x_k)\in (\mZ_2)^k,
\\ \textbf{y}_0& =  (y_1,\cdots,y_{k-1})\in (\mZ_2)^{k-1},
\\  \textbf{x}_i & =  (x_{i,1},\cdots,x_{i,m_i})\in (\mZ_2)^{m_i}, ~i=1,\cdots,k,
 \\  \textbf{y}_i & =  (y_{i,1},\cdots,y_{i,m_i-1})\in (\mZ_2)^{m_i-1},~ i=1,\cdots,k,
 \\   \textbf{z}_1 & = (z_{1,1}\cdots,z_{1,m_1+1})\in (\mZ_2)^{m_1+1},
\\   \textbf{z}_i & = (z_{i,0},z_{i,1},\cdots,z_{i,m_i+1})\in (\mZ_2)^{m_i+2},~i=2,\cdots,k-1,
\\  \textbf{z}_{k} &= (z_{k,0},z_{k,1}, \cdots,z_{k-1,m_k})\in (\mZ_2)^{m_k+1},
\end{split}
\end{eqnarray}
and $X=(\textbf{x}_0,\cdots,\textbf{x}_k),Y=(\textbf{y}_0, \cdots,\textbf{y}_k), Z=(\textbf{z}_1,\cdots,\textbf{z}_k)$.
By Propositions  \ref{jin1}-\ref{jin4}, for any $m_1\geq 1,\cdots,m_k\geq 1$,   the overlap matrix of
 $H_{m_1,\cdots,m_k}$ has
 this    form:  %$\Delta_{m_1,\cdots,m_k}^{X,Y, Z}$
$$
\Delta_{m_1,\cdots,m_k}^{X,Y, Z}=
\begin{bmatrix}
  \mathcal{L}_k^{\textbf{x}_0,\textbf{y}_0} & M_1^T&M_2^T&
\cdots  & M_{k-1}^T & M_k^T
\\ M_1 & \mathcal{L}_{m_1}^{\textbf{x}_1,\textbf{y}_1} & &
&&
\\ M_2 &&  \mathcal{L}_{m_2}^{\textbf{x}_2,\textbf{y}_2}
&  &\normalsize{\textbf{0}}&
\\ \vdots&&& \ddots & &
\\ M_{k-1} & &\normalsize{\textbf{0}}&&\mathcal{L}_{m_k}^{\textbf{x}_{k-1},\textbf{y}_{k-1}} &
\\ M_k & &&&&\mathcal{L}_{m_k}^{\textbf{x}_k,\textbf{y}_k}
\end{bmatrix},
$$
where $M_1$ is a $m_1\times k$  matrix   given by
\renewcommand\arraystretch{1.00}
\setlength{\arraycolsep}{1.00pt}
$$
\small{
\begin{bmatrix}
  z_{1,1}&0& 0&  \cdots &0
     \\  z_{1,2}&0& 0&  \cdots &0
       \\  \vdots &\vdots& \vdots &\vdots&\vdots
        \\  z_{1,{m_1-1}} &0& 0&  \cdots &0
       \\  z_{1,m_1} &z_{2,0}&  0& \cdots &0
\end{bmatrix}.
}
$$
For $2\leq i\leq k-1$, $M_i$ is a  $m_i\times k$  matrix with    the $i-1,i,i+1$
columns   given by the following  three  columns, respectively,
\renewcommand\arraystretch{1.00}
\setlength{\arraycolsep}{1.00pt}
\small{\begin{eqnarray*}
  \left(
  \begin{matrix}
  & z_{i-1,m_{i-1}+1}
  \\
  &0
  \\
  &\vdots
  \\
  & 0
  \\
  &0
  \end{matrix}
  \right),
   \left(
  \begin{matrix}
  & z_{i,1}
  \\
  &z_{i,2}
  \\
  &\vdots
  \\
  & z_{i,m_i-1}
  \\
  &z_{i,m_i}
  \end{matrix}
  \right),
     \left(
  \begin{matrix}
  & 0
  \\
  &0
  \\
  &\vdots
  \\
  & 0
  \\
  &z_{i+1,0}
  \end{matrix}
  \right);
\end{eqnarray*}
}
the other columns   of  $M_i$  are   $\textbf{0};$
$M_k$ is a $m_k\times k$ matrix   given by
\renewcommand\arraystretch{1.00}
\setlength{\arraycolsep}{1.00pt}
$$
\small{
\begin{bmatrix}
  0 & \cdots & 0 &z_{k-1,m_{k-1}+1} &z_{k,1}
     \\  0 &\cdots & 0& 0  &z_{k,2}
       \\  \vdots &\vdots& \vdots &\vdots
        \\  0  &\cdots & 0& 0  &z_{k,m_k-1}
       \\  0  &\cdots &0 & 0  &z_{k,m_k}
\end{bmatrix}.
}
$$
In this paper,   the \textit{transpose} of  $M_i$ is denoted by $M_i^T$.
%and we will use $M_i, M^{T}_i$ instead of $M^{Z}_i, (M^{Z}_i)^T$ for  the sake of simplicity.

We can also write  $\Delta_{m_1,\cdots,m_k}^{X,Y, Z}$
 in the this  form:
\renewcommand\arraystretch{1.00}
\setlength{\arraycolsep}{1.00pt}
$$
\tiny{
\left[
    \begin{matrix}
     \begin{matrix}
&&&& &  0
\\ &&&& & 0
\\ &&  \normalsize{\text{$\mathcal{L}$}}_{k-1}^{\textbf{x}_0,\textbf{y}_0} && & \vdots
\\ &&  &&&  0
\\ &&&& & y_{k-1}
\\ 0 & 0 &  \cdots & \text{0~~}  &\text{~~~~~$y_{k-1}$~~~~~}& \text{~$x_k$~}
\end{matrix}
  &  \normalsize{\text{$M^{T}_1$}}&  \cdots  &\normalsize{\text{$M^{T}_{k-1}$}} &
  \begin{matrix}
  0 &  0 &  \cdots & 0 &0 &0
     \\  0 & 0&\cdots & 0& 0  &0
       \\  \vdots &\vdots& \vdots &\vdots &\vdots &\vdots
             \\  0 & 0  &\cdots &0 & 0  &0
       \\  z_{k-1,m_{k-1}+1} & 0  &\cdots &0 & 0  &0
       \\  z_{k,1} & z_{k,2} &\cdots &z_{k,m_k-2} & z_{k,m_k-1}  &\text{~~$z_{k,m_k}$~~}
\end{matrix}
 \\
 \normalsize{\text{$M_1$}} &   \normalsize{\text{$\mathcal{L}$}}_{m_1}^{\textbf{x}_1,\textbf{y}_1}&&
 \\ \vdots  && \ddots & &  \normalsize{\textbf{0}}
 \\ \normalsize{\text{$M_{k-1}$}}  &&& \normalsize{\text{$\mathcal{L}$}}_{m_{k-1}}^{\textbf{x}_{k-1},\textbf{y}_{k-1}}&
 \\
\begin{matrix}
  0 & 0&  ~~\cdots~~ & 0 &z_{k-1,m_{k-1}+1} &z_{k,1}
     \\  0 & 0& \cdots & 0& 0  &z_{k,2}
       \\  \vdots&\vdots  &\vdots& \vdots &\vdots &\vdots
           \\  0  & 0 &\cdots &0 & 0  &z_{k,m_k-2}
       \\  0  & 0 &\cdots &0 & 0  &z_{k,m_k-1}
       \\  0  &0& \cdots &0 & 0  &z_{k,m_k}
\end{matrix}
 &    &\normalsize{\textbf{0}} & &
\begin{matrix}
  ~~~x_{k,1}~~~&y_{k,1}&&&&
     \\   y_{k,1} &\ddots~~& \ddots&&&
     \\ & \ddots ~~~& \ddots &\ddots
       \\     &  & \ddots~ & x_{k,m_k-2} &y_{k,m_k-2}&
       \\   & &  &y_{k,m_k-2}& x_{k,m_k-1}  &y_{k,m_k-1}
       \\   &  && &  y_{k, {m_k-1}}  & x_{k,m_k}
\end{matrix}
\end{matrix}
\right]
}
$$

Actually,  for $i\in \{1,\cdots,k\}$,
\begin{itemize}
  \item $x_{i}=1$  if and only if  the  edge  $e_i$ is twisted.
  \item $1\leq \ell\leq m_i$,   $x_{i,\ell}=1$  if and only if  the  edge  $e_{i,\ell}$ is twisted.
\end{itemize}

For $i\in \{2,\cdots,k\}$,
\begin{itemize}
  \item $z_{i,0}=1$  if and only if  the  edges  $e_i$ and $e_{i-1,m_{i-1}}$ overlap.
\end{itemize}
For $i\in \{1,\cdots,k-1\}$,
\begin{itemize}
  \item  $z_{i,m_i+1}=1$  if and only if  the  edges  $e_i$ and $e_{i+1,1}$ overlap,
  \item  $y_i=1$  if and only if  the  edges  $e_i$ and $e_{i+1}$ overlap.
\end{itemize}
For $i\in \{1,\cdots,k\} $ and
\begin{itemize}
  \item  $1\leq \ell\leq m_i$,   $z_{i,\ell}=1 $ if  and only if  the   edges  $e_i$ and $e_{i,\ell}$ overlap,
  \item    $1\leq \ell\leq m_i-1 $,   $ y_{i,\ell}=1$  if and only if  the  edges  $e_{i,\ell}$ and $e_{i,\ell+1}$ overlap.
\end{itemize}

\begin{theorem}\label{jin6}
Consider the cubic caterpillar-Halin graph  $H_{m_1,\cdots,m_k}$.  For any  fixed overlap
matrix $\Delta_{m_1,\cdots,m_k}^{X,Y, Z}$  of graph  $H_{m_1,\cdots,m_k}$, there are exactly 2 different $T$-rotation systems of
 $H_{m_1,\cdots,m_k}$  corresponding to that matrix.
 \end{theorem}

\begin{proof}
The values of   $\textbf{y}_0,\cdots,\textbf{y}_k,\textbf{z}_1,\cdots,\textbf{z}_k$ in (\ref{1-2}) can be  determined according to the matrix $\Delta_{m_1,\cdots,m_k}^{X,Y, Z}$.
We make the following discussions.

\begin{itemize}
  \item Case 1:  $z_{1,1}=1$.   If we color the vertex $v_{0,0}$ in Figure \ref{1-11}  black, by Propositions \ref{jin1},\ref{jin3}, we  need to  color the vertex  $v_{1,1}$  white.   Since the values of
  $\textbf{y}_0,\cdots,\textbf{y}_k,\textbf{z}_1,\cdots,\textbf{z}_k$ are given, by recursions and Propositions  \ref{jin1}-\ref{jin4}, we    color all  the rest of vertices.   Thus,  all the rotations of vertices in $H_{m_1,\cdots,m_k}$ are  determined.
For the case we color  the vertex $v_{0,0}$ white,   all the rotations of vertices in $H_{m_1,\cdots,m_k}$ are   determined by a  similar way.

  \item Case 2: $z_{1,1}=0$.  The discussions   are  similar  and we omit   the details.
\end{itemize}
Since we can   color  the rooted vertex  black or white, we finish our  proof.
\end{proof}

%\begin{remark}
%  In this paper, we consider  the Halin graph $H_{m_1,\cdots,m_k}$.  We always  assume   $m_i\geq 1, i=1,\cdots,k,$ so that  the overlap matrix of
% $H_{m_1,\cdots,m_k}$ has the  form $\Delta_{m_1,\cdots,m_k}^{X,Y, Z}$.
%\end{remark}
Let  $\delta_{m_1,m_2,\cdots,m_k}(z)
=\sum_{j=0}^{k+m_1+\cdots+m_k}\delta_{m_1,m_2,\cdots,m_k}[j]z^j
$,
where $\delta_{m_1,m_2,\cdots,m_k}[j]$
 is  the number of matrices of the form $\Delta_{m_1,\cdots,m_k}^{X,Y,Z}$
for  which this matrix  has rank $j$.
By Theorem \ref{jin6},  we have the following relation between $\E_{{m_1,\cdots,m_k}}(z) $ and
$\delta_{m_1,\cdots,m_k}(z).$
\begin{theorem}
\label{aa-1}
We have
$$\E_{{m_1,\cdots,m_k}}(z)
=2\delta_{m_1,\cdots,m_k}(z),$$
where $\E_{{m_1,\cdots,m_k}}(z)$ is   the Euler-genus polynomial of  $H_{m_1,\cdots,m_k}$.
\end{theorem}

For   $X=(\textbf{x}_0,\cdots,\textbf{x}_k),Y=(\textbf{y}_0,\cdots,\textbf{y}_k),
Z=(\textbf{z}_1,\cdots,\textbf{z}_{k-1})$,
when the  $k$-th row and $k$-th column of  matrix $\Delta_{m_1,\cdots,m_k}^{X,Y, Z}$  are  $\textbf{0},$
we denote this matrix   by  $\Lambda_{m_1,\cdots,m_k}^{X,Y,Z}$.
We  can also  write the   matrix $\Lambda_{m_1,\cdots,m_k}^{X,Y,Z}$ in the following form.
{\renewcommand\arraystretch{1}
\setlength{\arraycolsep}{1pt}

$$
\scriptsize{
\left[
    \begin{array}{ccccc}
     \begin{array}{ccccc}
&&& &  0
\\ &\normalsize{\text{$\mathcal{L}_{k-1}^{\textbf{x}_0,\textbf{y}_0}$}}  && & \vdots
\\ &  &&&  0
\\ &&& & 0
\\ 0  &  \cdots &  0  &\text{~~~~~~~0~~~~~~~}& 0
\end{array}
  & \normalsize{\text{$M^{T}_1$}}&  \cdots  &\normalsize{\text{$M^{T}_{k-1}$}} &
  \begin{array}{ccccc}
  0 &  0 &  \cdots &0 &0
       \\  \vdots &\vdots& \vdots &\vdots &\vdots
             \\  0 & 0  &\cdots &0   &0
       \\  z_{k-1,m_{k-1}+1} & 0  &\cdots &0  &0
       \\  0 & ~~0~~ &\cdots &~~~~0~~~~  &~~~~0~~~~
\end{array}
 \\
 \normalsize{\text{$M_1$}} &   \normalsize{\text{$\mathcal{L}$}}_{m_1}^{\textbf{x}_1,\textbf{y}_1}&&
 \\ \vdots  && \ddots & &  \normalsize{\textbf{0}}
 \\ \normalsize{\text{$M_{k-1}$}}  &&& \normalsize{\text{$\mathcal{L}$}}_{m_{k-1}}^{\textbf{x}_{k-1},\textbf{y}_{k-1}}&
 \\
\begin{array}{ccccc}
  0 &   ~~~\cdots~~~ & 0 &\tiny{z_{k-1,m_{k-1}+1}} &0
     \\  0 & \cdots & 0& 0  &0
       \\  \vdots &\vdots& \vdots &\vdots &\vdots
           \\  0  &\cdots &0 & 0  &0
       \\  0  & \cdots &0 & 0  &0
\end{array}
 &    & \normalsize{\textbf{0}} & &
\begin{array}{ccccc}
  \text{~~~~~$x_{k,1}$~~~~}&y_{k,1}&&&
     \\   \text{~~$y_{k,1}$~} &\ddots~& ~\ddots&&
     \\  &  \ddots ~~& \ddots&\ddots
       \\     & & \ddots~~ &x_{k,{m_k-1}}  &y_{k,m_k-1}
       \\   &  & & y_{k,{m_k-1}}  &x_{k,m_k}
\end{array}
\end{array}
\right]
}
$$
Actually, $\Lambda_{m_1,\cdots,m_k}^{X,Y,Z}$ is the overlap matrix for  graph $H_{n_1,\cdots,n_k}-e_k$  which is obtained by deleting the edge $e_k$ in  the  graph $H_{n_1,\cdots,n_k}$.

For any $m_1\geq 1,\cdots,m_k\geq 1$, we define the following polynomial related to $\Lambda_{m_1,\cdots,m_k}^{X,Y,Z}$,
$$\lambda_{m_1,m_2,\cdots,m_k}(z)=\sum_{j=0}^{k+m_1+\cdots+m_k}\lambda_{m_1,m_2,\cdots,m_k}[j]z^j,$$
where $\lambda_{m_1,m_2,\cdots,m_k}[j]$ is
the number of matrices of the form $\Lambda_{m_1,\cdots,m_k}^{X,Y,Z}$
for  which this matrix  has rank $j$.

In \cite{COZ11},  Chen et al. obtained   the following   overlap matrix  for   Ringel ladder $R_{m-1}$:
\renewcommand\arraystretch{1.15}
\setlength{\arraycolsep}{1.25pt}
\begin{align}
 \Phi_{m+1}^{\textbf{x},\textbf{y},\textbf{z}}= \begin{bmatrix}
     x_0 & z_1  & z_2 & z_3 &\cdots &z_{n-1} &z_n   \\
     z_1 & x_1  & y_1
     \\
    z_2  & y_1 &x_2 &y_2 & &  \normalsize{\textbf{0}}
        \\
     z_3  & &y_2  &x_3 &\ddots
    \\   \vdots  & &   & \ddots &   \ddots &  y_{m-2}
    \\   z_{m-1}  & &  \normalsize{\textbf{0}}  &   &  y_{m-2}&  x_{m-1}&y_{m-1}
     \\   z_{m}  & &   &   & &  y_{m-1}&x_{m}
    \end{bmatrix}.
\end{align}
where  $\textbf{x}=(x_0,x_1,\cdots,x_m)\in \mZ_2^{m+1}, \textbf{y}=(y_1,\cdots,y_{m-1})\in \mZ_2^{m-1}$, $\textbf{z}=(z_1,\cdots,z_{m-1},z_m)\in \mZ_2^{m}$.
Let
$$
\Phi_{m+1}=\{\Phi_{m+1}^{\textbf{x},\textbf{y},\textbf{z}}:\textbf{x}\in \mZ_2^{m+1}, \textbf{y}\in \mZ_2^{m-1}, \textbf{z}\in \mZ_2^{m}   \}
$$
and $\phi_{m+1}(z)=\sum_{j=0}^{m+1}\phi_{m+1}[j]z^j$,
where $\phi_{m+1}[j]$
is the number of matrices in $\Phi_{m+1}$ with  rank $j$.
%that  is the number of different assignment of the variables
%\begin{eqnarray*}
%x_\ell ,z_{m},y_{p},&&~~ \ell\in \{0,1,\cdots,n\}, m\in \{1,2,\cdots,n\}, p\in \{1,\cdots,n-1\}
%\end{eqnarray*}
%  for which the matrix $\Phi_{n+1}^{\textbf{x},\textbf{y},\textbf{z}}$  has rank $j$.

By  the fact  $H_{m_1,0} \cong R_{m_1}$  and  Proposition \ref{2-11}-Proposition \ref{2-12}, one sees that
\begin{eqnarray}
\label{zhazha4}
  \delta_{m_1,0}(z)=\phi_{m_1+2}(z)\quad \text{and}\quad \delta_{m_1,1}(z)=\phi_{m_1+3}(z).
\end{eqnarray}

  In this article,  the following previously derived results are useful.
\begin{proposition}(\cite{CMZ16})
\label{pp-2}
The following  recurrence relation holds for  $L_n(z), n\geq 3$,
  \begin{eqnarray*}
    L_{m}(z)=(1+2z)L_{m-1}(z)+4z^2 L_{m-2}(z)
  \end{eqnarray*}
and its   initial conditions  are  given by   $L_{1}(z)= z+1, L_{2}(z)= 4z^2+3z+1.$
\end{proposition}

\begin{proposition} (\cite{COZ11})
\label{2-13}
  The following  recurrence relation holds for   $\phi_m(z), m\geq 3$
  \begin{eqnarray*}
    \phi_{m+1}(z)=(1+4z)\phi_{m}(z)+16z^2 \phi_{m-1}(z)+2^mz^2 L_{m-1}(z)
  \end{eqnarray*}
  with    initial conditions  $\phi_{2}(z)= 4z^2+3z+1, \phi_{3}(z)=28z^3+28z^2+7z+1.$
    Furthermore, we have
    \begin{eqnarray*}
    && \phi(t,z):=\sum_{m\geq 2}\phi_m(z)t^m
    =\frac{t^2\big(1+3z+4z^2-2(1+5z+4z^2+2z^3)t
   -16z^2(2+6z+5z^2)t^2-168z^4(1+z)t^3\big)}{(1-2t-4tz-16z^2t^2)(1-t-4tz-16z^2t^2)}.
    \end{eqnarray*}
    \end{proposition}

%\begin{eqnarray*}
%  g_k(L_{n-1})=\Bigg\{
%  \begin{split}
%   &  2^{n-2+k}*{ n-k \choose k}* \frac{2n-3k}{n-k},  \quad when ~~0\leq k\leq \lfloor \frac{n}{2}\rfloor
%    \\
%    & 0,  \quad \quad \quad \quad\quad \quad \quad \quad\quad \quad \quad \quad\quad  otherwise.
%  \end{split}
%\end{eqnarray*}

\section{Euler-genus polynomials $\E_{m_1,\cdots,m_k}(z)$ for graph  $H_{m_1,\cdots,m_k}$}
In subsections  \ref{7-2} and \ref{7-3},  we  give the computations of $\E_{{m_1,m_{2}}}(z)$
and
 $\E_{{m_1,\cdots,m_k}}(z)$ respectively.

\subsection{The computation of $\cE_{{m_1,m_2}}(z)$.}
\label{7-2}
By definition,   $\Delta_{m_1,m_2}^{X,Y,Z}$ is   a matrix given by
\renewcommand\arraystretch{1.02}
\setlength{\arraycolsep}{1.02pt}
$$\left(
   \begin{array}{ccccccccccc}
     x_1 & y_1  &z_{1,1}& z_{1,2}&\cdots  & z_{1,m_1-1} &z_{1,m_1}&z_{1,m_1+1}   & 0& \cdots&0
     \\
     y_1 & x_2  & 0 & 0& \cdots &0&z_{2,0} &z_{2,1} &z_{2,2}&\cdots  & z_{2,m_2}
     \\
     z_{1,1} &0  &x_{1,1} & y_{1,1}
     \\
     z_{1,2} &0 &y_{1,1} &x_{1,2}&  \ddots
     \\
     \vdots & \vdots &&\ddots&\ddots &y_{1,m_1-2} &&&\normalsize{\textbf{0}}&
     \\  z_{1,m_1-1} & 0&&&  y_{1,m_1-2}& x_{1,m_1-1} & y_{1,m_1-1}
     \\
     z_{1,m_1} &z_{2,0}&& & &  y_{1,m_1-1}&x_{1,m_1}
     \\
     z_{1,m_1+1} &z_{2,1}&& &&&& x_{2,1}&y_{2,1}
       \\
     0 &z_{2,2}&& && & &y_{2,1}&x_{2,2} & \ddots
         \\
     \vdots  &\vdots && &\normalsize{\textbf{0}}& &&&\ddots &\ddots &  y_{2,m_2-1}
            \\
     0  &z_{2,m_2} &&&  &&&&&y_{2,m_2-1} &x_{2,m_2}
     \end{array}
 \right)
$$
By definition,  the matrix $\Lambda_{m_1,m_2}^{X,Y, Z}$ is  given by
\renewcommand\arraystretch{1.02}
\setlength{\arraycolsep}{1.02pt}
$$
    \begin{bmatrix}
     x_1 &0 &z_{1,1}& z_{1,2}&\cdots  & z_{1,m_1} &z_{1,m_1+1}&0   & \cdots&0
     \\
    0 & 0  & 0 & 0  &\cdots &0 &0 &0&\cdots  & 0
     \\
     z_{1,1} &0  &x_{1,1} & y_{1,1}&
     \\
     z_{1,2} &0 &y_{1,1} &x_{1,2}&\ddots &
     \\
     \vdots &\vdots&&\ddots&\ddots &y_{1,m_1-1}
     \\
     z_{1,m_1} &0&& &y_{1,m_1-1}&x_{1,m_1} &&&\normalsize{\textbf{0}}
     \\
     z_{1,m_1+1} &0&& &&&x_{2,1}&y_{2,1}
       \\
     0 &0&& &&&y_{2,1}&x_{2,2}& \ddots
         \\
     \vdots  &\vdots && &\normalsize{\textbf{0}}  &&&\ddots &\ddots &y_{2,m_2-1}
            \\
     0  &0 && &&&&&y_{2,m_2-1} &x_{2,m_2}
     \end{bmatrix}
$$

%Let us set
%$$
%\lambda_{m_1,0}(z)= \phi_{m_1+1}(z).
%$$
Now, we give a    lemma to demonstrate the    recurrence   relation for   polynomial  $\lambda_{m_1,m_2}(z)$.
\begin{lemma}
\label{11-1}
  For any $m_1\geq 1,m_2\geq 2$,    the following    recurrence   relation holds  for  $\lambda_{m_1,m_2}(z)$
    \begin{eqnarray}\label{1-130}
    \lambda_{m_1,m_2}(z)&=& (1+2z)\lambda_{m_1,m_2-1}(z)+4z^2 \lambda_{m_1,m_2-2}(z)
   % \\ &&\quad +\phi_{m_1+1}(z)\big[L_{m_2}(z)-(1+2z)L_{m_2-1}(z)-4z^2L_{m_2-2}(z)\big]
  \end{eqnarray}
  with initial conditions $\lambda_{m_1,0}(z)= \phi_{m_1+1}(z),\lambda_{m_1,1}(z)= (1+2z)\phi_{m_1+1}(z)+2^{m_1+1}z^2L_{m_1}(z)$
\end{lemma}

\begin{proof}

There are $4$  different choices   of values for
$(x_{2,m_2},y_{2,m_2-1}),m_2\geq 3.$

\begin{center}
\begin{tabular}{|c|c|}
\hline Cases  & Contributions to $ \lambda_{m_1,m_2}(z)$\\
\hline  $1:(x_{2,m_2},y_{2,m_2-1})=(0,0)  $&$\lambda_{m_1,m_2-1}(z)  $\\
\hline $2: (x_{2,m_2},y_{2,m_2-1})=(1,0)  $&$z\lambda_{m_1,m_2-1}(z)   $ \\
\hline $3: (x_{2,m_2}, y_{2,m_2-1})=(0,1)  $&$4z^2\lambda_{m_1,m_2-2}(z) $\\
\hline $4:(x_{2,m_2},y_{2,m_2-1})=(1,1)  $&$z \lambda_{m_1,m_2-1}(z)  $
\\ \hline
\end{tabular}
\end{center}
Combining cases 1-4,  one arrives at  (\ref{1-130}).

 By direct computation, one sees that
 \begin{eqnarray*}
  \left\{
  \begin{split}
    \lambda_{m_1,1}(z)&= (1+2z)\phi_{m_1+1}(z)+2^{m_1+1}z^2L_{m_1}(z),
    \\ \lambda_{m_1,2}(z)&=(1+2z)\lambda_{m_1,1}(z)+4z^2\phi_{m_1+1}(z).
  \end{split}
  \right.
  \end{eqnarray*}
  So, this theorem holds with $m_2=2.$

\end{proof}

We have  the    recurrence   relation for    $\cE_{{m_1,m_2}}(z)$.

\begin{theorem}
\label{11-3}
   For any $m_1\geq 1,m_2\geq 2$,    the following    recurrence   relation holds  for  $\cE_{{m_1,m_2}}(z)$,
  \begin{eqnarray}
  \label{2-1}
    \cE_{{m_1,m_2}}(z) &=& (1+4z)\cE_{{m_1,m_2-1}}(z)+ 16z^2\cE_{{m_1,m_2-2}}(z)
    + 2^{m_2+3}z^2\lambda_{m_1,m_2-1}(z)
  \end{eqnarray}
  with initial conditions $\cE_{{m_1,0}}(z)=2\phi_{m_1+2}(z),\cE_{{m_1,1}}(z)=2 \phi_{m_1+3}(z).$
\end{theorem}
\begin{proof}
 The initial conditions come from (\ref{zhazha4}) and Theorem \ref{aa-1}.

Consider  the overlap matrix  $\Delta_{m_1,m_2}^{X,Y,Z}$ of
 graph $H_{m_1,m_{2}}$.
Since there  are eight different choices   of values for  $(x_{2,m_2}, y_{2,m_2-1},z_{2,m_2})$,
\begin{center}
\begin{tabular}{|c|c|}
\hline Cases  & Contributions to $ \delta_{m_1,m_2}(z)$\\
\hline  $1:(x_{2,m_2},y_{2,m_2-1},z_{2,m_2})=(0,0,0) $&$\delta_{m_1,m_2-1}(z) $\\
\hline $2: (x_{2,m_2}, y_{2,m_2-1},z_{2,m_2})=(1,0,0) $&$z\delta_{m_1,m_2-1}(z)  $ \\
\hline $3: (x_{2,m_2}, y_{2,m_2-1},z_{2,m_2})=(0,1,0) $&$8z^2\delta_{m_1,m_2-2}(z)$\\
\hline $4:(x_{2,m_2}, y_{2,m_2-1},z_{2,m_2})=(0,0,1) $&$2^{m_2+2}z^2\lambda_{m_1,m_2-1}(z) $\\
\hline  $5:( x_{2,m_2}, y_{2,m_2-1},z_{2,m_2})=(0,1,1) $&$8z^2\delta_{m_1,m_2-2}(z)  $\\
\hline $6: (x_{2,m_2}, y_{2,m_2-1}, z_{2,m_2})=(1,1,0)  $&$z\delta_{m_1,m_2-1}(z)   $ \\
\hline $7:  (x_{2,m_2}, y_{2,m_2-1}, z_{2,m_2})=(1,0,1) $&$z\delta_{m_1,m_2-1}(z) $\\
\hline $8:(x_{2,m_2},  y_{2,m_2-1},z_{2,m_2})=(1,1,1)  $&$z\delta_{m_1,m_2-1}(z)  $
\\ \hline
\end{tabular}
\end{center}
 one arrives at
  \begin{eqnarray*}
    \delta_{m_1,m_2}(z)&=& (1+4z)\delta_{m_1,m_2-1}(z)+16z^2\delta_{m_1,m_2-2}(z)
    +2^{m_2+2}z^2\lambda_{m_1,m_2-1}(z).
  \end{eqnarray*}
  Combing the above equality with Theorem \ref{aa-1}, the proof is completed.
\end{proof}

\subsection{The computation of $\E_{{m_1,\cdots,m_k}}(z)$.}
\label{7-3}

%We define
%$\lambda_{m_1,\cdots,m_{k-1},0}(z) =\delta_{m_1,\cdots,m_{k-1}}(z).$
Before we  give a recurrence   relation for  $\E_{{m_1,\cdots,m_k}}(z)$,   we introduce a lemma.
\begin{lemma}
\label{5-2}
The following    recurrence   relation holds  for  the polynomial  $\lambda_{m_1,m_2,\cdots,m_k}(z)$ $(k\geq 3, m_1,\cdots,m_{k-1}\geq	1,m_k\geq 2)$,
\begin{eqnarray}
\label{1-8}
\begin{split}
  \lambda_{m_1,m_2,\cdots,m_{k-1}, m_k}(z)&=(1+2z) \lambda_{m_1,m_2,\cdots,m_{k-1},m_k-1}(z)+4z^2 \lambda_{m_1,m_2,\cdots,m_{k-1},m_k-2}(z)
  \end{split}
\end{eqnarray}
  with the initial conditions
\begin{eqnarray*}
&& \lambda_{m_1,\cdots,m_{k-1},0}(z) =\frac{1}{2}\E_{{m_1,\cdots,m_{k-1}}}(z),
\\ && \lambda_{m_1,\cdots,m_{k-1},1}(z) =\frac{1+2z}{2}\E_{{m_1,\cdots,m_{k-1}}}(z) +2^{m_{k-1}+3}z^2\cdot \lambda_{m_1,\cdots,m_{k-1}}(z).
\end{eqnarray*}
\end{lemma}
\begin{proof}
 There are $4$ different combinations of values for the variables $(x_{k,m_k},y_{k,m_k-1}):$
\begin{center}
\begin{tabular}{|c|c|}
\hline Cases  & Contributions to $ \lambda_{m_1,m_2,\cdots,m_k}(z)$\\
\hline  $1:(x_{k,m_k},y_{k,m_k-1})=(0,0)  $&$\lambda_{m_1,m_2,\cdots,m_{k-1},m_k-1}(z)  $\\
\hline $2: (x_{k,m_k}, y_{k,m_k-1})=(1,0)  $&$z\lambda_{m_1,m_2,\cdots,m_{k-1},m_k-1}(z)  $ \\
\hline $3: (x_{k,m_k}, y_{k,m_k-1})=(0,1)  $&$4z^2\lambda_{m_1,m_2,\cdots,m_{k-1},m_k-2}(z) $\\
\hline $4:(x_{k,m_k}, y_{k,m_k-1})=(1,1)  $&$z \lambda_{m_1,m_2,\cdots,m_{k-1},m_k-1}(z)  $
\\ \hline
\end{tabular}
\end{center}
Combining cases 1-4, (\ref{1-8}) holds with $m_k\geq 3$.

By direct calculating, one sees  that
\begin{eqnarray*}
\lambda_{m_1,m_2,\cdots,m_{k-1},1}(z) &=& (1+2z)\delta_{m_1,\cdots,m_{k-1}}(z)+2^{m_{k-1}+3}z^2\cdot \lambda_{m_1,\cdots,m_{k-1}}(z)
\\ &=& \frac{1+2z}{2}\E_{{m_1,\cdots,m_{k-1}}}(z)+2^{m_{k-1}+3}z^2\cdot \lambda_{m_1,\cdots,m_{k-1}}(z),
\\
\lambda_{m_1,m_2,\cdots,m_{k-1},2}(z)&=&(1+2z) \lambda_{m_1,m_2,\cdots,m_{k-1},1}(z)+4z^2 \delta_{m_1,\cdots,m_{k-1}}(z)
\\ &=&(1+2z) \lambda_{m_1,m_2,\cdots,m_{k-1},1}(z)+2z^2 \E_{{m_1,\cdots,m_{k-1}}}(z)
\end{eqnarray*}
which implies  (\ref{1-8}) holds with $m_k=2$, that is the initial conditions hold.
\end{proof}

We  have the following   recurrence   relation for   $\E_{{m_1,\cdots,m_k}}(z)$.
\begin{theorem}
\label{5-1}
  The polynomial  $\E_{{m_1,\cdots,m_k}}(z)(m_1,\cdots, m_{k-1}\geq 1,m_k\geq 2 )$ satisfies  the recurrence relation
  \begin{eqnarray}
 \label{1-1}
  \begin{split}
   \E_{{m_1,\cdots,m_k}}(z)&= (1+4z)\E_{{m_1,\cdots,m_{k-1},m_k-1}}(z)+16z^2 \E_{{m_1,\cdots,m_{k-1},m_k-2}}(z)
    \\ &+2^{m_k+3}z^2\lambda_{m_1,\cdots,m_{k-1},m_k-1}(z)
    \end{split}
  \end{eqnarray}
  with the initial conditions $\E_{{m_1,\cdots,m_{k-1},0}}(z)=\E_{{m_1,\cdots,m_{k-1}+1}}(z),
 \E_{{m_1,\cdots,m_{k-1},1}}(z) =\E_{{m_1,\cdots,m_{k-1}+2}}(z).$
\end{theorem}

Before the  proof of Theorem  \ref{5-1},
  we explain here  how     to calculate  $\E_{{m_1,\cdots,m_k}}(z)$.
 \begin{remark}
  Fix  $m_1,\cdots,m_k.$
    We will  use iteration to  calculate
     \begin{eqnarray}
     \label{x-1}
     \lambda_{m_1,\cdots,m_{k-1},m_k}(z),~~\cE_{m_1,\cdots,m_{k-1}, m_k}(z),~~\cE_{m_1,\cdots,m_{k-1},m_k+1}(z),~~\cE_{m_1,\cdots,m_{k-1},  m_k+2}(z).
   \end{eqnarray}

 First,   in  Lemma \ref{11-1} and Theorem  \ref{11-3}, we get  the values of
   \begin{eqnarray*}
     \lambda_{m_1,m_2}(z),~~\cE_{m_1,m_2}(z),~~\cE_{m_1,m_2+1}(z),~~\cE_{m_1,m_2+2}(z).
   \end{eqnarray*}
   Assume that we have already get the values of (\ref{x-1}) for $k\leq \ell-1.$
 Using   Lemma \ref{5-2},  the values of  $\lambda_{m_1,\cdots,m_{\ell-1},i}(z), 0\leq i \leq m_\ell$ are obtained.  Then, by  Theorem  \ref{5-1}, one arrives at  the values of
      \begin{eqnarray*}
     \cE_{m_1,\cdots,m_{\ell-1}, m_{\ell}}(z),~~\cE_{m_1,\cdots,m_{\ell-1},m_{\ell}+1}(z),~~\cE_{m_1,\cdots,m_{\ell-1},  m_{\ell}+2}(z).
   \end{eqnarray*}
By the above arguments, we  can  get the explicit expression for $\cE_{m_1,\cdots,m_{k-1}, m_k}(z)$.
 \end{remark}

\subsection{Proof of Theorem  \ref{5-1}}
%Now, we  give a proof of Theorem  \ref{5-1}.

\begin{proof}
Noting   $H_{m_1,\cdots,m_{k-1},0}=H_{m_1,\cdots,m_{k-1}+1} $  and  Proposition\ref{2-11}-Proposition \ref{2-12}, one sees  that  the initial conditions hold.
So, we only need to prove  (\ref{1-1}).

Consider  the overlap matrix  $\Delta_{m_1,\cdots,m_k}^{X,Y,Z}$ of
 graph $H_{m_1,\cdots,m_k}$.
 There are the following  eight different combinations of values for the variables  $x_{k,m_k}, y_{k,m_k-1}$ and $z_{k,m_k}$.

Case 1: $(x_{k,m_k}, y_{k,m_k-1},z_{k,m_k})=(0,0,0)$. One   easily  sees  that this case contributes to  $\delta_{m_1,m_2,\cdots,m_k}(z)$ by a term  $\delta_{m_1,m_2,\cdots,m_k-1}(z).$

Case 2:$(x_{k,m_k}, y_{k,m_k-1}, z_{k,m_k})=(0,0,1)$.
If  $z_{k,0}=1,$ we add row $-1$ to row $-(m_k+1)$ and then add column $-1$ to column$-(m_k+1)$ . Here row(column)  $-\ell$ denotes the  row $\ell$ counting  from the last row(column).
Making  similar operations  for   $z_{k,1},\cdots, z_{k,m_k-1}$ and  $x_k,y_{k-1},$
the    resulting matrix is given by
\renewcommand\arraystretch{1.2}
\setlength{\arraycolsep}{1.20pt}
$$
\left[
    \begin{array}{ccccc}
\tiny{
     \begin{array}{cccccc}
&&&& &  0
\\ &&&& & 0
\\ && \normalsize{\text{$\mathcal{L}_{k-1}^{\textbf{x}_0,\textbf{y}_0}$}}   & &  & \vdots
\\ &&  &&&  0
\\ &&&& & 0
\\ 0 & 0 &  \cdots & 0  & \text{~~~~~~~~$0$~~~~~~~}& ~~0~~
\end{array}
}
  & \normalsize{\text{$M^{T}_1$}}&  \cdots  &\normalsize{\text{$M^{T}_{k-1}$}} &
  \tiny{
  \begin{array}{ccccccc}
  0 &  0 &  \cdots & 0 &0 &0
     \\  0 & 0& ~~\cdots~~ & 0& 0  &0
       \\  \vdots &\vdots& ~~\vdots~~ &\vdots &\vdots &\vdots
             \\  0 & 0  &\cdots &0 & 0  &0
       \\  z_{k-1,m_{k-1}+1} & ~~0~~  &\cdots &0 & 0  &0
       \\  0 & ~~0~~ &\cdots &\text{~~~~$0$~~~~} & \text{~~~~$0$~~~~} &\text{~~$1$~~}
\end{array}
}
 \\
 \normalsize{\text{$M_1$}} &   \normalsize{\text{$\mathcal{L}$}}_{m_1}^{\textbf{x}_1,\textbf{y}_1}&&
 \\ \vdots  && \ddots & &  \normalsize{\textbf{0}}
 \\
 \\ \normalsize{\text{$M_{k-1}$}}  &&& \normalsize{\text{$\mathcal{L}$}}_{m_{k-1}}^{\textbf{x}_{k-1},\textbf{y}_{k-1}}&
 \\
 \tiny{
\begin{array}{cccccc}
  0 & 0&  ~~~\cdots~~~ & 0 & z_{k-1,m_{k-1}+1}  &\text{~~$0$~~}
     \\  0 & 0& \cdots & 0& 0  &0
       \\  \vdots&\vdots  &\vdots& \vdots &\vdots &\vdots
           \\  0  & 0 &\cdots &0 & 0  &0
       \\  0  & 0 &\cdots &0 & 0  &0
       \\  0  &0& \cdots &0 & 0  &1
\end{array}
}
 &    & \normalsize{\textbf{0}} & &
 \tiny{
\begin{array}{cccccc}
  \text{~~~~$x_{k,1}$~~~~~}&y_{k,1}&&&&
     \\   \text{~~~~$y_{k,1}$~~~~~} &x_{k,2}& ~~\ddots&&&
    \\   &\ddots~~& \ddots&  \ddots
    \\  && \ddots&\ddots~~& y_{k,m_k-2}&
       \\      & && y_{k,m_k-2}&  x_{k,m_k-1} &\text{~~$0$~~}
       \\   &  && &  0   &\text{~~$0$~~}
\end{array}
}
\end{array}
\right]
$$
Since there are $2^{m_k+2}$ choices of values for  $ z_{k,0},\cdots, z_{k,m_k-1}$ and $x_k,y_{k-1}$,  this case contributes a term
$2^{m_k+2}z^2\lambda_{m_1,m_2,\cdots,m_{k-1},m_k-1}(z).$

  Case 3:  $(x_{k,m_k}, y_{k,m_k-1},z_{k,m_k})=(1,0,0)$.
  This case contributes  a term   $z\delta_{m_1,m_2,\cdots,m_k-1}(z).$

   Case 4: $(x_{k,m_k}, y_{k,m_k-1},z_{k,m_k})=(0,1,0)$.
If  $z_{k,m_k-1}=1,$ we add row $-1$ to row $k$,
then we add the column $-1$  to the  column $k$.
We make    similar operations  for   $x_{k,m_k-1},y_{k,m_k-2}$
 and  get  the following matrix.

\renewcommand\arraystretch{1.30}
\setlength{\arraycolsep}{1.30pt}
$$
\tiny{
\left[
    \begin{matrix}
     \begin{matrix}
&&&& &  0
\\ &&&& & 0
\\ &&\normalsize{\text{$\mathcal{L}_{k-1}^{\textbf{x}_0,\textbf{y}_0}$}}  && & \vdots
\\ &&  &&&  0
\\ &&&& & y_{k-1}
\\ 0 & 0 &  \cdots & \text{0~~}  &\text{~~~~$y_{k-1}$~~~~}& x_k
\end{matrix}
  & \normalsize{\text{$M^{T}_1$}}&  \cdots  &\normalsize{\text{$M^{T}_{k-1}$}} &
  \begin{matrix}
  0 &  0 &  \cdots & 0 &0 &0
     \\  0 & 0& \cdots & 0& 0  &0
       \\  \vdots &\vdots& \vdots &\vdots &\vdots &\vdots
             \\  0 & 0  &\cdots &0 & 0  &0
       \\  z_{k-1,m_{k-1}+1} & 0  &\cdots &0 & 0  &0
       \\  z_{k,1} & z_{k,2} &~~~\cdots~~~ &z_{k,m_k-2} & \text{~~~~$0$~~~~}   &\text{~~~~$0$~~~~}
\end{matrix}
 \\
 \normalsize{\text{$M_1$}} &   \normalsize{\text{$\mathcal{L}$}}_{m_1}^{\textbf{x}_1,\textbf{y}_1}&&
 \\ \vdots  && \ddots & &  \normalsize{\textbf{0}}
 \\ \normalsize{\text{$M_{k-1}$}}  &&& \normalsize{\text{$\mathcal{L}$}}_{m_{k-1}}^{\textbf{x}_{k-1},\textbf{y}_{k-1}}&
 \\
 \\
\begin{matrix}
  0 & 0&  ~~~\cdots~~~ & 0 &\tiny{z_{k-1,m_{k-1}+1}} &~~z_{k,1}~~
     \\  0 & 0& \cdots & 0& 0  &z_{k,2}
       \\  \vdots&\vdots  &\vdots& \vdots &\vdots &\vdots
           \\  0  & 0 &\cdots &0 & 0  &z_{k,m_k-2}
       \\  0  & 0 &\cdots &0 & 0  &0
       \\  0  &0& \cdots &0 & 0  &0
\end{matrix}
 &    & \normalsize{\textbf{0}} & &
\begin{matrix}
  \text{~~~~~~$x_{k,1}$~~~~~~}&y_{k,1}&&&&
     \\   y_{k,1} &x_{k,2}& \ddots&&&
     \\  &\ddots~&\ddots &y_{k,m_k-3}&&
       \\      & &y_{k,m_k-3}&  x_{k,m_k-2} &\text{~~~$0$~~~~~}&
       \\   & &&\text{~~~$0$~~~~~~~}&\text{~~~$0$~~~~~}   &\text{~~~$1$~~~~~~~}
       \\   &  && & \text{~~~$1$~~~~~}   &\text{~~~$0$~~~~~~~}
\end{matrix}
\end{matrix}
\right]
}
$$
Since there are $2^3$  different choices  of values for the  variables  $x_{k,m_k-1},y_{k,m_k-2},z_{k,m_k-1},$  it   contributes  a term  $8z^2\delta_{m_1,m_2,\cdots,m_{k-1},m_k-2}(z).$

Case 5: $(x_{k,m_k}, y_{k,m_k-1},z_{k,m_k})=(0,1,1)$.
We  add row $-2$ to row $k$ and  add column $-2$ to column   $k$.
We get the following matrix.

\renewcommand\arraystretch{1.20}
\setlength{\arraycolsep}{1.20pt}
{$$
\tiny{
\left[
    \begin{matrix}
     \begin{matrix}
&&&& &&   0
\\ &&&& & & 0
\\ &&&& & & 0
\\ &&& \normalsize{\text{$\mathcal{L}_{k-1}^{\textbf{x}_0,\textbf{y}_0}$}}  & & & \vdots
\\ &&  &&& &  0
\\ &&&& &&  y_{k-1}
\\ ~0~ & ~0~ &  ~0~& \cdots & \text{0}  &\text{~~~~~$y_{k-1}$~~~~~}& \text{~~$x_k$~~}
\end{matrix}
  & \normalsize{\text{$M^{T}_1$}}&  \cdots  &\normalsize{\text{$M^{T}_{k-1}$}} &
  \begin{matrix}
  0 &  0 &  \cdots & 0 &0 &0 & 0
  \\ 0 &  0 &  \cdots & 0 &0 &0 & 0
     \\  0 & 0& \cdots &0&  0& 0  &0
       \\  \vdots &\vdots& \vdots &\vdots &\vdots &\vdots & \vdots
             \\  0 & 0  &\cdots & 0&0 & 0  &0
       \\  z_{k-1,m_{k-1}+1} & 0  &\cdots & 0 &0 & 0  &0
       \\  z_{k,1} & z_{k,2} &\cdots & z_{k,m_k-3}&z_{k,m_k-2}^*  &z_{k,m_k-1}^* &\text{~~~~$0$~~~~}
\end{matrix}
 \\
\normalsize{\text{$M_1$}} &   \normalsize{\text{$\mathcal{L}$}}_{m_1}^{\textbf{x}_1,\textbf{y}_1}&&
 \\ \vdots  && \ddots & &  \normalsize{\textbf{0}}
 \\ \normalsize{\text{$M_{k-1}$}}  &&& \normalsize{\text{$\mathcal{L}$}}_{m_{k-1}}^{\textbf{x}_{k-1},\textbf{y}_{k-1}}&
 \\
 \\
\begin{matrix}
 ~ 0~ & ~0~& ~0~&   ~~~\cdots~~~ & 0 &\tiny{z_{k-1,m_{k-1}+1}} &z_{k,1}
     \\  0 & 0& 0&  \cdots & 0& 0  &z_{k,2}
       \\  \vdots&\vdots  &\vdots& \vdots &\vdots &\vdots &\vdots
        \\  0  & 0&0  &\cdots &0 & 0  &z_{k,m_k-3}
           \\  0  & 0 &0  &\cdots &0 & 0  &z_{k,m_k-2}^*
       \\  0  & 0 & 0&\cdots &0 & 0  &z_{k,m_k-1}^*
       \\  0  &0& 0&  \cdots &0 & 0  &0
\end{matrix}
 &    & \normalsize{\textbf{0}} & &
\begin{matrix}
  \text{~~~~$x_{k,1}$~~~}&y_{k,1}&&&&&
     \\    y_{k,1} &x_{k,2}&\ddots &&&&
    \\   &\ddots~~& \ddots& \ddots  &&&
     \\   && \ddots& \ddots &y_{k,m_k-3}&&
    \\  && &y_{k,m_k-3}&x_{k,m_k-2}&  y_{k,m_k-2}&
       \\      & && & y_{k,m_k-2}&  x_{k,m_k-1}& \text{~~~$1$~~~}
       \\   &  && & &   1  &\text{~~~$0$~~~}
\end{matrix}
\end{matrix}
\right]
}
$$}
where $z_{k,m_k-2}^*=z_{k,m_k-2}+y_{k,m_k-2} $ and $z_{k,m_k-1}^*=z_{k,m_k-1}+x_{k,m_k-1}.$
If  $y_{k,m_k-2}=1$, we  add row $-1$ to row $-3$ and then add column $-1$ to column$-3$.
We make   similar discussions  for   $x_{k,m_k-1}, z_{k,m_k-1}^*$
and  transform  the matrix  $\Delta_{m_1,\cdots,m_k}^{X,Y, Z}$ into the following form
\renewcommand\arraystretch{1.20}
\setlength{\arraycolsep}{1.20pt}
$$
\tiny{
\left[
    \begin{matrix}
     \begin{matrix}
&&&& &&   0
\\ &&&& & & 0
\\ &&&& & & 0
\\ &&& \normalsize{\text{$\mathcal{L}_{k-1}^{\textbf{x}_0,\textbf{y}_0}$}}  & & & \vdots
\\ &&  &&& &  0
\\ &&&& &&  y_{k-1}
\\ ~0~ & ~0~ &  ~0~& \cdots & \text{0}  &\text{~~~~~$y_{k-1}$~~~~~}& \text{~~$x_k$~~}
\end{matrix}
  & \normalsize{\text{$M^{T}_1$}}&  \cdots  &\normalsize{\text{$M^{T}_{k-1}$}} &
  \begin{matrix}
  0 &  0 &  \cdots & 0 &0 &0 & 0
  \\ 0 &  0 &  \cdots & 0 &0 &0 & 0
     \\  0 & 0& \cdots &0&  0& 0  &0
       \\  \vdots &\vdots& \vdots &\vdots &\vdots &\vdots & \vdots
             \\  0 & 0  &\cdots & 0&0 & 0  &0
       \\  z_{k-1,m_{k-1}+1} & 0  &\cdots & 0 &0 & 0  &0
       \\  z_{k,1} & z_{k,2} &\cdots & z_{k,m_k-3}&z_{k,m_k-2}^*  &0 &\text{~~~~$0$~~~~}
\end{matrix}
 \\
\normalsize{\text{$M_1$}} &   \normalsize{\text{$\mathcal{L}$}}_{m_1}^{\textbf{x}_1,\textbf{y}_1}&&
 \\ \vdots  && \ddots & &  \normalsize{\textbf{0}}
 \\ \normalsize{\text{$M_{k-1}$}}  &&& \normalsize{\text{$\mathcal{L}$}}_{m_{k-1}}^{\textbf{x}_{k-1},\textbf{y}_{k-1}}&
 \\
 \\
\begin{matrix}
 ~ 0~ & ~0~& ~0~&   ~~~\cdots~~~ & 0 &\tiny{z_{k-1,m_{k-1}+1}} &z_{k,1}
     \\  0 & 0& 0&  \cdots & 0& 0  &z_{k,2}
       \\  \vdots&\vdots  &\vdots& \vdots &\vdots &\vdots &\vdots
        \\  0  & 0&0  &\cdots &0 & 0  &z_{k,m_k-3}
           \\  0  & 0 &0  &\cdots &0 & 0  &z_{k,m_k-2}^*
       \\  0  & 0 & 0&\cdots &0 & 0  &0
       \\  0  &0& 0&  \cdots &0 & 0  &0
\end{matrix}
 &    & \normalsize{\textbf{0}} & &
\begin{matrix}
  \text{~~~~$x_{k,1}$~~~}&y_{k,1}&&&&&
     \\    y_{k,1} &x_{k,2}&\ddots &&&&
    \\   &\ddots~~& \ddots& \ddots  &&&
     \\   && \ddots& \ddots &y_{k,m_k-3}&&
    \\  && &y_{k,m_k-3}&x_{k,m_k-2}&  0&
       \\      & && & \text{$0$}&  \text{$0$} & \text{~~~$1$~~~}
       \\   &  && & &   1  &\text{~~~$0$~~~}
\end{matrix}
\end{matrix}
\right]
}
$$
Since there are $2^3$ possible choices of values for   $ x_{k,m_k-1},y_{k,m_k-2},z_{k,m_k-1}$,  this case  contributes   a term  $8z^2\delta_{m_1,m_2,\cdots,m_{k-1},m_k-2}(z).$

Case 6: $(x_{k,m_k}, y_{k,m_k-1},z_{k,m_k})=(1,1,0)$.
We  add row $-1$ to row $-2$ and  add column $-1$ to column$-2$.
Then, one  sees  that it   contributes   a term  $z\delta_{m_1,m_2,\cdots,m_{k-1},m_k-1}(z)$.

Case 7: $(x_{k,m_k}, y_{k,m_k-1},z_{k,m_k})=(1,0,1)$.
We  add row $-1$ to row $k$ and then add column $-1$ to column$k$.
it  contributes   a term  $z\delta_{m_1,m_2,\cdots,m_{k-1},m_k-1}(z)$.

Case 8: $(x_{k,m_k}, y_{k,m_k-1},z_{k,m_k})=(1,1,1)$.
 We  add row $-1$ to the rows  $-2,k$ and then add column $-1$ to the columns   $-2,k$ respectively.
Then,  it   contributes   a term  $z\delta_{m_1,m_2,\cdots,m_{k-1},m_k-1}(z)$.

Combining cases 1-8, we obtain
  \begin{eqnarray*}
  \begin{split}
    \delta_{m_1,m_2,\cdots,m_k}(z)&= (1+4z)\delta_{m_1,\cdots,m_{k-1},m_k-1}(z)+16z^2 \delta_{m_1,\cdots,m_{k-1},m_k-2}(z)
    \\ &+2^{m_k+2}z^2\lambda_{m_1,\cdots,m_{k-1},m_k-1}(z).
    \end{split}
  \end{eqnarray*}
  Combing the above equality with Theorem \ref{aa-1},  one gets    (\ref{1-1}).
\end{proof}

Recall that the definition of $\eps_{m_1,\cdots,m_k}(j)$ is given in subsection 1.3.
\begin{theorem}\label{torus}
For any $k\geq 2$ and positive integers  $m_1,\cdots,m_k,$  we have $\eps_{m_1,\cdots,m_k}(0)=2$
and $\eps_{m_1,\cdots,m_k}(1) =8(m_1+m_2+\cdots+m_k+k)-10.$
\end{theorem}
 \label{Q-1}
\begin{proof}  For any $m_1,\cdots,m_k\geq 1,$ obviously, we have  $\eps_{m_1,\cdots,m_k}(0)=2.$   We use  induction on $k$ to  prove
  \begin{eqnarray}
    \label{1-21}
     \eps_{m_1,\cdots,m_k}(1)& =&8(m_1+m_2+\cdots+m_k+k)-10.
  \end{eqnarray}

 First, we claim that,  for $k=2,$ (\ref{1-21}) holds.
 By Proposition \ref{2-13}, one sees  that
 \begin{eqnarray*}
 \left\{
 \begin{split}
   \phi_{n+1}[1]&=\phi_n[1]+4,
   \\
   \phi_2[1]&=3
   \end{split}
   \right.
 \end{eqnarray*}
 which yields
 \begin{eqnarray}
 \label{V-3}
 \phi_{n}[1] =3+4(n-2)=4n-5.
 \end{eqnarray}
 Furthermore, using   Theorem \ref{2-1}, we have
 \begin{eqnarray*}
   \eps_{m_1,m_2}(1)=\eps_{m_1,m_2-1}(1)+8,
   \quad  \eps_{m_1,1}(1)=2\phi_{m_1+3}[1] =8m_1+14,
 \end{eqnarray*}
 which  implies that
$ \eps_{m_1,m_2}(1)=\eps_{m_1,1}(1)+8(m_2-1)=8m_1+14+8(m_2-1)=8(m_1+m_2+2)-10.
 $ Therefore, (\ref{1-21}) holds with $k=2.$

 Assume    (\ref{1-21}) holds  for any $m_1,\cdots,m_{k-1}\geq 1$.
 By  Theorem \ref{5-1}, one easily sees that
 \begin{eqnarray*}
 \left\{
 \begin{split}
   \eps_{m_1,\cdots,m_{k-1},m_k}(1)&= \eps_{m_1,\cdots,m_{k-1},m_k-1}(1)+8,
   \\
   \eps_{m_1,\cdots,m_{k-1},1}(1)&=\eps_{m_1,\cdots,m_{k-1}+2}(1).
   \end{split}
   \right.
 \end{eqnarray*}
Furthermore, by induction, we have
 \begin{eqnarray*}
    && \eps_{m_1,\cdots,m_{k-1},m_k}(1) = \eps_{m_1,\cdots,m_{k-1},1}(1)+8(m_k-1)
   = \eps_{m_1,\cdots,m_{k-1}+2}(1) +8(m_k-1)
    \\ && = 8(m_1+\cdots+m_{k-2}+m_{k-1}+2+k-1)-10+8(m_k-1)
    \\ && = 8(m_1+\cdots+m_{k-2}+m_{k-1}+m_k+k)-10
 \end{eqnarray*}
  which  completes   the proof  of    (\ref{1-21}).
\end{proof}
With  a similar method, one has the following results.
\begin{theorem}\label{doubletorus}
For any  $k\geq 2$ and  $m_1,\cdots,m_k\geq 1,$  $ \eps_{m_1,\cdots,m_k}(2)=8 \sum_{i=1}^k\Big(-9 + 4 i + (-3 + 4 i) m_i + 2 m_i^2   \Big)+2\sum_{i=2}^k ( 2^{m_i+3}+3\cdot 2^{m_{i-1}+3} )-2^{m_1+5}+32\sum_{i=2}^k \sum_{j=1}^{i-1}(m_jm_i+m_j).$
\end{theorem}

\section{ The generating functions  for graphs $H_{m_1,\cdots,m_k}.$ }

For $i=0,1,2,3$, we define
\begin{eqnarray*}
   \lambda^i(t_1,t_2,\cdots,t_k,z)&=&\sum_{m_1,m_2,\cdots,m_{k-1}\geq 1,m_k\geq i}2^{m_k} \lambda_{m_1,m_2,\cdots,m_k}(z)t_1^{m_1}t_2^{m_2}\cdots t_k^{m_k}
     \\
  \E^i(t_1,t_2,\cdots,t_k,z)&=&\sum_{m_1,m_2,\cdots,m_{k-1}\geq 1, m_k\geq i} \E_{m_1,m_2,\cdots,m_k}(z)t_1^{m_1}t_2^{m_2}\cdots t_k^{m_k}
    \end{eqnarray*}
and
$
L^* (t,z)= \sum_{m\geq 1}2^m L_m(z)t^m.
$

In subsection \ref{10-4}, we give the computations of $\lambda^1(t_1,t_2,z),\lambda^2(t_1,t_2,z), \E^1(t_1,t_2,z),$ and $\E^2(t_1,t_2,z)$.
In subsection \ref{10-1}, we  list  some lemmas.
These lemmas demonstrate the relations between  $$ \lambda^i(t_1,t_2,\cdots,t_k,z),i=0,1,2,~~~\E^i(t_1,t_2,\cdots,t_k,z),i=0,1,2,3.$$
 Using  these lemmas  in subsections  \ref{10-4},\ref{10-1},     we   compute the following    functions in subsection \ref{10-2},
  $$
  \lambda^1(t_1,\cdots, t_k,z),\lambda^2(t_1,t_2,\cdots,t_k,z), \E^1(t_1,t_2,\cdots,t_k,z),\E^2(t_1,\cdots, t_k,z).
  $$

\subsection{The computations of $\lambda^1(t_1,t_2,z),\lambda^2(t_1,t_2,z), \E^1(t_1,t_2,z),\E^2(t_1,t_2,z)$}
\label{10-4}

First, one easily sees   the following lemma hold.
\begin{lemma}
\label{pz-10}
one has
$L^* (t,z)=\frac{8z^2t^2+2t(1+z)}{-16z^2t^2-(4z+2)t+1}.$
\end{lemma}

 The  expressions of   $(\lambda^1(t_1,t_2,z),\lambda^2(t_1,t_2,z))$ and  $(\E^1(t_1,t_2,z),\E^2(t_1,t_2,z))$ are given by the following two lemmas .
\begin{lemma}\label{lema:add}
We have
\begin{eqnarray*}
\lambda^1(t_1,t_2,z)&=&\frac{(2t_2+4zt_2+16z^2t_2^2)t_1^{-1}\phi({t_1,z})
+4z^2t_2L^*(t_1,z)}{1-2t_2-4zt_2-16z^2t_2^2},
\\
\lambda^2(t_1,t_2,z)&=& \lambda^1(t_1,t_2,z)-2t_2t_1^{-1}(2z+1)\phi(t_1,z)-4z^2t_2L^*(t_1,z),
\end{eqnarray*}
where the function $\phi$ is given by Proposition \ref{2-13}.
\end{lemma}
\begin{proof}

By Lemma    \ref{11-1}, we have
%\begin{eqnarray*}
%\begin{split}
%  2^{m_2}\lambda_{m_1,m_2}(z)&=2^{m_2}(1+2z) \lambda_{m_1,m_2-1}(z)+2^{m_2}4z^2 \lambda_{m_1,m_2-2}(z).
%  \end{split}
%\end{eqnarray*}
%Therefore, we have
\begin{eqnarray*}
\sum_{m_1\geq 1,m_2\geq 2}
t_1^{m_1}t_2^{m_2}2^{m_2}\lambda_{m_1,m_2}(z)&=&\sum_{m_1\geq 1,m_2\geq 2} t_1^{m_1}t_2^{m_2} \cdot 2^{m_2}(1+2z) \lambda_{m_1,m_2-1}(z)
\\ && \quad +\sum_{m_1\geq 1,m_2\geq 2}t_1^{m_1}t_2^{m_2} \cdot  2^{m_2+2}z^2 \lambda_{m_1,m_2-2}(z).
\end{eqnarray*}
That is
\begin{eqnarray*}
&& \lambda^1(t_1,t_2,z)-\sum_{m_1\geq 1}
t_1^{m_1}t_2\cdot 2^{1}\lambda_{m_1,1}(z)
\\ &&=(1+2z)2t_2 \lambda^1(t_1,t_2,z)+16z^2t_2^2\sum_{m_1\geq 1,m_2\geq 0}t_1^{m_1}t_2^{m_2} \cdot  2^{m_2} \lambda_{m_1,m_2}(z)
\\ &&=(1+2z)2t_2 \lambda^1(t_1,t_2,z)+16z^2t_2^2\Big[\lambda^1(t_1,t_2,z)+ \sum_{m_1\geq 1 }    \lambda_{m_1,0}(z)t_1^{m_1}\Big].
\end{eqnarray*}
Therefore, by Lemma \ref{11-1}, one sees that
\begin{eqnarray*}
 \lambda^1(t_1,t_2,z)&=& \frac{1}{1-2t_2-4zt_2-16z^2t_2^2}\Big[2t_2\sum_{m_1\geq 1}
t_1^{m_1}\lambda_{m_1,1}(z)+ 16z^2t_2^2\sum_{m_1\geq 1 }   \lambda_{m_1,0}(z)t_1^{m_1}\Big]
\\ &=& \frac{1}{1-2t_2-4zt_2-16z^2t_2^2}\Big[2t_2\sum_{m_1\geq 1}
t_1^{m_1}\Big((2z+1)\phi_{m_1+1}(z)+2^{m_1+1}z^2L_{m_1}(z)\Big)
\\ &&\quad\quad \quad \quad \quad \quad\quad \quad\quad \quad \quad  \quad \quad  + 16z^2t_2^2\sum_{m_1\geq 1 }   \phi_{m_1+1}(z)t_1^{m_1}\Big]
\\ &=& \frac{1}{1-2t_2-4zt_2-16z^2t_2^2}\Big[(2t_2+4zt_2+16z^2t_2^2)t_1^{-1}\phi({t_1,z})
+4z^2t_2L^*(t_1,z)\Big].
\end{eqnarray*}
Now, we consider the computation of $\lambda^2(t_1,t_2,z)$. By Lemma \ref{11-1}, it holds that
\begin{eqnarray*}
\lambda^2(t_1,t_2,z)&=& \lambda^1(t_1,t_2,z)-\sum_{m_1\geq 1}
t_1^{m_1}t_2\cdot 2\lambda_{m_1,1}(z)
\\ &=& \lambda^1(t_1,t_2,z)-2t_2\sum_{m_1\geq 1}
t_1^{m_1}\Big((2z+1)\phi_{m_1+1}(z)+2^{m_1+1}z^2L_{m_1}(z)\Big)
\\ &=& \lambda^1(t_1,t_2,z)-2t_2t_1^{-1}(2z+1)\phi(t_1,z)-4z^2t_2L^*(t_1,z).
\end{eqnarray*}
\end{proof}

\begin{lemma}
We have
\begin{eqnarray*}
\E^1(t_1,t_2,z)&=&
\frac{(32z^2t_1^{-2}t_2^2+2t_1^{-3}t_2)\phi(t_1,z)+16z^2t_2\lambda^1(t_1,t_2,z)}{1-t_2-4zt_2-16z^2t_2^2}
\\ && -\frac{(8z^2+6z+2)t_1^{-1}t_2+(56z^3+56z^2+14z+2)t_2+32z^2t_2^2(4z^2+3z+1)}{1-t_2-4zt_2-16z^2t_2^2},
\\
\E^2(t_1,t_2,z)&=&  \E^1(t_1,t_2,z)- 2t_1^{-3}t_2\phi(t_1,z)+(8z^2+6z+2)t_1^{-1}t_2+(56z^3+56z^2+14z+2)t_2.
\end{eqnarray*}
\end{lemma}
\begin{proof}
By  Theorem
\ref{11-3},
  one sees that
  \begin{eqnarray}
 \nonumber  &&\E^2(t_1,t_2,z)=\sum_{m_1\geq 1,m_2\geq 2}\E_{m_1,m_2}(z)t_1^{m_1}t_2^{m_2}
  \\   \nonumber  &=&   \sum_{m_1\geq 1,m_2\geq 2}(1+4z)\E_{m_1,m_2-1}(z)t_1^{m_1}t_2^{m_2}
  + \sum_{m_1\geq 1,m_2\geq 2}16z^2\E_{m_1,m_2-2}(z) t_1^{m_1}t_2^{m_2}
  \\   \nonumber  &&+  \sum_{m_1\geq 1,m_2\geq 2}2^{m_2+3}z^2\lambda_{m_1,m_2-1}(z)t_1^{m_1}t_2^{m_2}
  \\   \nonumber  &=& (1+4z)t_2\E^1(t_1,t_2,z)+16z^2t_2^2\big[\E^1(t_1,t_2,z)+\sum_{m_1\geq 1}\E_{m_1,0}(z)t_1^{m_1}\big]
+16z^2t_2 \lambda^1(t_1,t_2,z)
    \\  \nonumber   &=& (1+4z)t_2\E^1(t_1,t_2,z)+16z^2t_2^2\big[\E^1(t_1,t_2,z)+2\sum_{m_1\geq 1}\phi_{m_1+2}(z)t_1^{m_1}\big]
+16z^2t_2 \lambda^1(t_1,t_2,z)
      \\   \nonumber  &=& 16z^2t_2^2\Big[\E^1(t_1,t_2,z)+2t_1^{-2}\big(\phi(t_1,z)-(4z^2+3z+1)t_1^2\big)\Big]
 +(1+4z)t_2\E^1(t_1,t_2,z)\\ &&\nonumber \quad +16z^2t_2 \lambda^1(t_1,t_2,z)
     \\ \nonumber   &=& (t_2+4zt_2+16z^2t_2^2)\E^1(t_1,t_2,z)+32z^2t_2^2\big(\phi(t_1,z)t_1^{-2}-(4z^2+3z+1)\big)
  \\  \label{16-1}&&\quad +16z^2t_2 \lambda^1(t_1,t_2,z).
  \end{eqnarray}
  On the other hand, it also holds that
  \begin{eqnarray}
  \nonumber &&\cE^2(t_1,t_2,z)= \sum_{m_1\geq 1,m_2\geq 2}\E_{m_1,m_2}(z)t_1^{m_1}t_2^{m_2}
  \\ \nonumber &=&     \E^1(t_1,t_2,z)- \sum_{m_1\geq 1}\E_{m_1,1}(z)t_1^{m_1}t_2
  =     \E^1(t_1,t_2,z)- 2\sum_{m_1\geq 1}\phi_{m_1+3}(z)t_1^{m_1}t_2
     \\ \nonumber  &=&     \E^1(t_1,t_2,z)- 2\sum_{m_1\geq -1}\phi_{m_1+3}(z)t_1^{m_1}t_2+2(4z^2+3z+1)t_1^{-1}t_2+2(28z^3+28z^2+7z+1)t_2
       \\  \label{zz-1}&=&    \E^1(t_1,t_2,z)- 2t_1^{-3}t_2\phi(t_1,z)+2(4z^2+3z+1)t_1^{-1}t_2+2(28z^3+28z^2+7z+1)t_2.
  \end{eqnarray}
  Combining (\ref{zz-1}) with (\ref{16-1}), we finish the  proof.
\end{proof}
\subsection{Some lemmas}
\label{10-1}
In this subsection, we list some lemmas  used in subsection 4.3.
\begin{lemma}
\label{12-1}
For $k\geq 3,$ we have $\lambda^{1}(t_1,t_2,\cdots,t_k,z)=\lambda^{0}(t_1,t_2,\cdots,t_k,z)-\frac{\E^1(t_1,t_2,\cdots,t_{k-1},z)}{2}.$

\end{lemma}
\begin{proof}
By Lemma \ref{5-2}, we obtain
    \begin{eqnarray*}
  \lambda^{1}(t_1,t_2,\cdots,t_k,z)&=&\sum_{m_1,m_2,\cdots,m_{k-1}\geq 1,m_k\geq 1}2^{m_k} \lambda_{m_1,m_2,\cdots,m_k}(z)t_1^{m_1}t_2^{m_2}\cdots t_k^{m_k}
\\ &=&\lambda^{0}(t_1,t_2,\cdots,t_k,z)-\sum_{m_1,m_2,\cdots,m_{k-1}\geq 1} \lambda_{m_1,m_2,\cdots,m_{k-1},0}(z)t_1^{m_1}t_2^{m_2}\cdots t_{k-1}^{m_{k-1}}
\\&=&\lambda^{0}(t_1,t_2,\cdots,t_k,z) -\sum_{m_1,m_2,\cdots,m_{k-1}\geq 1}\frac{\E_{m_1,m_2,\cdots,m_{k-1}}(z)}{2}t_1^{m_1}t_2^{m_2}\cdots t_{k-1}^{m_{k-1}}
\\ &=& \lambda^{0}(t_1,t_2,\cdots,t_k,z)-\frac{\E^1(t_1,t_2,\cdots,t_{k-1},z)}{2}.
  \end{eqnarray*}

\end{proof}

\begin{lemma}
\label{12-2}
For any $k\geq 3,$ we have
$\lambda^{2}(t_1,t_2,\cdots,t_k,z)
  =\lambda^{1}(t_1,t_2,\cdots,t_k,z)
  -t_k(1+2z)\E^1(t_1,\cdots$ $,t_{k-1},z)-16z^2t_k\lambda^1(t_1,\cdots,t_{k-1,z}).
$
\end{lemma}
\begin{proof}
With the help of  Lemma \ref{5-2}, we obtain
    \begin{eqnarray*}
  &&\lambda^{2}(t_1,t_2,\cdots,t_k,z)\\&=&\sum_{m_1,m_2,\cdots,m_{k-1}\geq 1,m_k\geq 2}2^{m_k} \lambda_{m_1,m_2,\cdots,m_k}(z)t_1^{m_1}t_2^{m_2}\cdots t_k^{m_k}
\\ &=&\lambda^{1}(t_1,t_2,\cdots,t_k,z)-2t_k\sum_{m_1,m_2,\cdots,m_{k-1}\geq 1} \lambda_{m_1,m_2,\cdots,m_{k-1},1}(z)t_1^{m_1}t_2^{m_2}\cdots t_{k-1}^{m_{k-1}}
\\&=&\lambda^{1}(t_1,t_2,\cdots,t_k,z)
\\ && -2t_k\sum_{m_1,m_2,\cdots,m_{k-1}\geq 1}\big[\frac{1+2z}{2}\E_{m_1,\cdots,m_{k-1}}(z)+2^{m_{k-1}+3}z^2\cdot \lambda_{m_1,\cdots,m_{k-1}}(z)\big]t_1^{m_1}t_2^{m_2}\cdots t_{k-1}^{m_{k-1}}
\\ &=& \lambda^{1}(t_1,t_2,\cdots,t_k,z)
-t_k(1+2z)\E^1(t_1,t_2,\cdots,t_{k-1},z)-16z^2t_k\lambda^1(t_1,t_2,\cdots,t_{k-1},z).
  \end{eqnarray*}

\end{proof}

\begin{lemma}\label{12-3}
For any $k\geq 3,$ we have
$$
 \lambda^2(t_1,t_2,\cdots,t_k,z)= (2+4z)t_k \lambda^1(t_1,t_2,\cdots,t_k,z)+16z^2t_k^2 \lambda^0(t_1,t_2,\cdots,t_k,z).
 $$
\end{lemma}
\begin{proof}
%Using  (\ref{1-8}), we have
%\begin{eqnarray*}
%\begin{split}
%  2^{m_k}\lambda_{m_1,m_2,\cdots,m_{k-1}, m_k}(z)&=2^{m_k}(1+2z) , \lambda_{m_1,m_2,\cdots,m_{k-1},m_k-1}(z)+2^{m_k}4z^2 \lambda_{m_1,m_2,\cdots,m_{k-1},m_k-2}(z).
%  \end{split}
%\end{eqnarray*}
Using  (\ref{1-8}), we have
\begin{eqnarray*}
&& \sum_{m_1,\cdots,m_{k-1}\geq 1,m_k\geq 2}
t_1^{m_1}t_2^{m_2}\cdots t_k^{m_k}\cdot 2^{m_k}\lambda_{m_1,m_2,\cdots,m_{k-1}, m_k}(z)
\\ &&=\sum_{m_1,\cdots,m_{k-1}\geq 1,m_k\geq 2} t_1^{m_1}t_2^{m_2}\cdots t_k^{m_k} \cdot 2^{m_k}(1+2z) \lambda_{m_1,m_2,\cdots,m_{k-1},m_k-1}(z)
\\ && \quad +\sum_{m_1,\cdots,m_{k-1}\geq 1,m_k\geq 2}t_1^{m_1}t_2^{m_2}\cdots t_k^{m_k} \cdot  2^{m_k}4z^2 \lambda_{m_1,m_2,\cdots,m_{k-1},m_k-2}(z).
\end{eqnarray*}
That is $\lambda^2(t_1,t_2,\cdots,t_k,z)= (2+4z)t_k \lambda^1(t_1,t_2,\cdots,t_k,z)+16z^2t_k^2\lambda^0(t_1,t_2,\cdots,t_k,z).$

\end{proof}

\begin{lemma}
\label{pz-2}
For any $k\geq 3,$ we have
   $$
  \E^2(t_1,t_2,\cdots,t_k,z)
  =(1+4z)t_k \E^1(t_1,t_2,\cdots,t_k,z)+16z^2t_k^2\E^0(t_1, \cdots,t_k,z)
+16z^2t_k \lambda^1(t_1,t_2,\cdots,t_k,z).
$$
\end{lemma}

\begin{proof}
  By Theorem  \ref{5-1}, we have
    \begin{eqnarray*}
     \E^2(t_1,t_2,\cdots,t_k,z) &=& \sum_{m_1,m_2,\cdots,m_{k-1}\geq 1,m_k\geq 2} \E_{m_1,m_2,\cdots,m_k}(z)t_1^{m_1}t_2^{m_2}\cdots t_k^{m_k}
\\    &=&\quad (1+4z) \sum_{m_1,m_2,\cdots,m_{k-1}\geq 1,m_k\geq 2} \E_{m_1,\cdots,m_{k-1},m_k-1}(z)t_1^{m_1}t_2^{m_2}\cdots t_k^{m_k}
\\ &&\quad  +16z^2 \sum_{m_1,m_2,\cdots,m_{k-1}\geq 1,m_k\geq 2} \E_{m_1,\cdots,m_{k-1},m_k-2}(z)t_1^{m_1}t_2^{m_2}\cdots t_k^{m_k}
\\ &&\quad +\sum_{m_1,m_2,\cdots,m_{k-1}\geq 1,m_k\geq 2} 2^{m_k+3}z^2\lambda_{m_1,\cdots,m_{k-1},m_k-1}(z)t_1^{m_1}t_2^{m_2}\cdots t_k^{m_k}
\\  &=& (1+4z)t_k \E^1(t_1,t_2,\cdots,t_k,z)+16z^2t_k^2\E^0(t_1,t_2,\cdots,t_k,z)
+16z^2t_k \lambda^1(t_1,t_2,\cdots,t_k,z)
  \end{eqnarray*}
  which completes the proof.
\end{proof}

\begin{lemma}
\label{pz-3}
For  $k\geq 3,$ one has
$$  \E^3(t_1,t_2,\cdots,t_k,z)
  = (1+4z)t_k \E^2(t_1,t_2,\cdots,t_k,z)+16z^2t_k^2\E^1(t_1, \cdots,t_k,z)+16z^2t_k \lambda^2(t_1,t_2,\cdots,t_k,z).
$$
\end{lemma}

\begin{proof}
  By Theorem  \ref{5-1}, one arrives at that
    \begin{eqnarray*}
  \E^3(t_1,t_2,\cdots,t_k,z) &=& \sum_{m_1,m_2,\cdots,m_{k-1}\geq 1,m_k\geq 3} \E_{m_1,m_2,\cdots,m_k}(z)t_1^{m_1}t_2^{m_2}\cdots t_k^{m_k}
    \\ &=&(1+4z) \sum_{m_1,m_2,\cdots,m_{k-1}\geq 1,m_k\geq 3} \E_{m_1,\cdots,m_{k-1},m_k-1}(z)t_1^{m_1}t_2^{m_2}\cdots t_k^{m_k}
    \\ && +16z^2 \sum_{m_1,m_2,\cdots,m_{k-1}\geq 1,m_k\geq 3} \E_{m_1,\cdots,m_{k-1},m_k-2}(z)t_1^{m_1}t_2^{m_2}\cdots t_k^{m_k}
    \\ &&+\sum_{m_1,m_2,\cdots,m_{k-1}\geq 1,m_k\geq 3} 2^{m_k+3}z^2\lambda_{m_1,\cdots,m_{k-1},m_k-1}(z)t_1^{m_1}t_2^{m_2}\cdots t_k^{m_k}
\\  &=& (1+4z)t_k \E^2(t_1,t_2,\cdots,t_k,z)+16z^2t_k^2\E^1(t_1,t_2,\cdots,t_k,z)
+16z^2t_k \lambda^{2}(t_1,t_2,\cdots,t_k,z).
  \end{eqnarray*}
\end{proof}

\begin{lemma}
\label{pz-1}
We have
$ \E^1(t_1,t_2,\cdots,t_k,z) = \E^0(t_1,t_2,\cdots,t_k,z) -t_{k-1}^{-1}\E^2(t_1,t_2,\cdots,t_{k-1},z).$
\end{lemma}
\begin{proof}
Noting these  initial conditions in Theorem \ref{5-1}, one  easily  sees that
    \begin{eqnarray*}
 &&  \E^1(t_1,t_2,\cdots,t_k,z)
  \\  &=& \sum_{m_1,m_2,\cdots,m_{k-1}\geq 1,m_k\geq 1} \E_{m_1,m_2,\cdots,m_k}(z)t_1^{m_1}t_2^{m_2}\cdots t_k^{m_k}
    \\ &=& \sum_{m_1,m_2,\cdots,m_{k-1}\geq 1,m_k\geq 0} \E_{m_1,m_2,\cdots,m_k}(z)t_1^{m_1}t_2^{m_2}\cdots t_k^{m_k}
    - \sum_{m_1,m_2,\cdots,m_{k-1}\geq 1} \E_{m_1,m_2,\cdots,m_{k-1},0}(z)t_1^{m_1}t_2^{m_2}\cdots t_{k-1}^{m_{k-1}}
    \\ &=&  \E^0(t_1,t_2,\cdots,t_k,z)- \sum_{m_1,m_2,\cdots,m_{k-1}\geq 1} \E_{m_1,m_2,\cdots,m_{k-1}+1}(z)t_1^{m_1}t_2^{m_2}\cdots t_{k-1}^{m_{k-1}}
    \\&=& \E^0(t_1,t_2,\cdots,t_k,z) -t_{k-1}^{-1}\E^2(t_1,t_2,\cdots,t_{k-1},z).
  \end{eqnarray*}
\end{proof}

\begin{lemma}
\label{pz-4}
We have $\E^2(t_1,\cdots, t_k,z)=\E^1(t_1,\cdots, t_k,z)-t_kt_{k-1}^{-2}\E^3(t_1,\cdots, t_{k-1},z).$
\end{lemma}
\begin{proof}
By these  initial conditions in Theorem \ref{5-1}, we obtain
  \begin{eqnarray*}
 \E^2(t_1,t_2,\cdots,t_k,z)
   &=& \sum_{m_1,m_2,\cdots,m_{k-1}\geq 1,m_k\geq 2} \E_{m_1,m_2,\cdots,m_k}(z)t_1^{m_1}t_2^{m_2}\cdots t_k^{m_k}
    \\ &=& \sum_{m_1,m_2,\cdots,m_{k-1}\geq 1,m_k\geq 1} \E_{m_1,m_2,\cdots,m_k}(z)t_1^{m_1}t_2^{m_2}\cdots t_k^{m_k}
   \\ &&- \sum_{m_1,m_2,\cdots,m_{k-1}\geq 1} \E_{m_1,m_2,\cdots,m_{k-1},1}(z)t_1^{m_1}t_2^{m_2}\cdots t_{k-1}^{m_{k-1}}t_k^1
    \\ &=&  \E^1(t_1,t_2,\cdots,t_k,z)- t_k \sum_{m_1,m_2,\cdots,m_{k-1}\geq 1} \E_{m_1,m_2,\cdots,m_{k-1}+2}(z)t_1^{m_1}t_2^{m_2}\cdots t_{k-1}^{m_{k-1}}
    \\&=& \E^1(t_1,t_2,\cdots,t_k,z) -t_k t_{k-1}^{-2}\E^3(t_1,t_2,\cdots,t_{k-1},z).
  \end{eqnarray*}
\end{proof}

\subsection{The calculation  of generating functions}
\label{10-2}

In this subsection, we will demonstrate      a  recurrence relation  between
$$(
\lambda^1(t_1,t_2,\cdots,t_k,z),
\lambda^2(t_1,t_2,\cdots,t_k,z),
\E^1(t_1,t_2,\cdots,t_k,z),
\E^2(t_1,t_2,\cdots,t_k,z))$$ and
$$(
\lambda^1(t_1,t_2,\cdots,t_{k-1},z),
\lambda^2(t_1,t_2,\cdots,t_{k-1},z),
\E^1(t_1,t_2,\cdots,t_{k-1},z),
\E^2(t_1,t_2,\cdots,t_{k-1},z)).$$
In achieve them, we demonstrate  two lemmas first.
\begin{lemma}
For $k\geq 3,$ we have
\begin{eqnarray}
\nonumber \lambda^1(t_1,t_2,\cdots,t_k,z)&=&A_{11}(t_k,z)\lambda^1(t_1,t_2,\cdots,t_{k-1},z)+A_{12}(t_k,z)\lambda^2(t_1,t_2,\cdots,t_{k-1},z)
\\ \label{5-3}&&+A_{13}(t_k,z)\E^1(t_1,t_2,\cdots,t_{k-1},z)+A_{14}(t_k,z)\E^2(t_1,t_2,\cdots,t_{k-1},z),
\\ \nonumber
\lambda^2(t_1,t_2,\cdots,t_k,z)&=&A_{21}(t_k,z)\lambda^1(t_1,t_2,\cdots,t_{k-1},z)+A_{22}(t_k,z)\lambda^2(t_1,t_2,\cdots,t_{k-1},z)
\\&& \label{5-4}+A_{23}(t_k,z)\E^1(t_1,t_2,\cdots,t_{k-1},z)+A_{24}(t_k,z)\E^2(t_1,t_2,\cdots,t_{k-1},z),
\end{eqnarray}
where
\begin{eqnarray*}
A_{11}(t_k,z)=\frac{16z^2t_k}{1-2t_k-4zt_k-16z^2t_k^2}, && \quad A_{12}(t_k,z)= 0,
\\ A_{13}(t_k,z)= \frac{t_k(1+2z)+8z^2t_k^2}{1-2t_k-4zt_k-16z^2t_k^2}, && \quad A_{14}(t_k,z)=0,
\\ A_{21}(t_k,z)=A_{11}(t_k,z)-16z^2t_k,   \,\,\,\,\,\,\,\,\,\,\,\,\,
 && \quad A_{22}(t_k,z)= 0,
\\ A_{23}(t_k,z)= A_{13}(t_k,z)-t_k(1+2z),\,
 && \quad A_{24}(t_k,z)=0.
\end{eqnarray*}

\end{lemma}

\begin{proof}
By Lemmas  \ref{12-1}, \ref{12-3}, one sees that
\begin{eqnarray*}
\lambda^2(t_1,t_2,\cdots,t_k,z)&=& (2t_k+4zt_k+16z^2t_k^2)\lambda^1(t_1,t_2,\cdots,t_k,z)+8z^2t_k^2\E^1(t_1,t_2,\cdots,t_{k-1},z).
\end{eqnarray*}
Combining the above equality with Lemma  \ref{12-2}, we get
\begin{eqnarray*}
&&\lambda^1(t_1,\cdots, t_k,z)\\&=& \frac{t_k(1+2z)\E^1(t_1,t_2,\cdots,t_{k-1},z)+16z^2t_k\lambda^1(t_1,t_2,\cdots,t_{k-1},z)+8z^2t_k^2\E^1(t_1,t_2,\cdots,t_{k-1},z)}{1-2t_k-4zt_k-16z^2t_k^2},
\end{eqnarray*}
which completes the proof of (\ref{5-3}). Using Lemma  \ref{12-2}   again, we get   (\ref{5-4}).
\end{proof}

\begin{lemma}
For $k\geq 3,$ we have
\begin{eqnarray}
\nonumber \E^1(t_1,t_2,\cdots,t_k,z)&=&A_{31}(t_k,z)\lambda^1(t_1,t_2,\cdots,t_{k-1},z)+A_{32}(t_k,z)\lambda^2(t_1,t_2,\cdots,t_{k-1},z)
\\ \label{6-1}&&+A_{33}(t_k,z)\E^1(t_1,t_2,\cdots,t_{k-1},z)+A_{34}(t_k,z)\E^2(t_1,t_2,\cdots,t_{k-1},z),
\end{eqnarray}
and
\begin{eqnarray}
\nonumber \E^2(t_1,t_2,\cdots,t_k,z)&=&A_{41}(t_k,z)\lambda^1(t_1,t_2,\cdots,t_{k-1},z)+A_{42}(t_k,z)\lambda^2(t_1,t_2,\cdots,t_{k-1},z)
\\ \label{6-2}&&+A_{43}(t_k,z)\E^1(t_1,t_2,\cdots,t_{k-1},z)+A_{44}(t_k,z)\E^2(t_1,t_2,\cdots,t_{k-1},z),
\end{eqnarray}
where
%\begin{eqnarray*}
%A_{11}(t_k,z)=\frac{8z^2t_k}{1-2t_k-8z^2t_k^2},  && \quad A_{12}(t_k,z)= 0,
%\\  A_{13}(t_k,z)= \frac{2t_k+8z^2t_k^2}{1-2t_k-8z^2t_k^2},
% && \quad  A_{14}(t_k,z)=0,
%\\  A_{21}(t_k,z)=A_{11}(t_k,z)-8z^2t_k, && \quad
%A_{22}(t_k,z)= 0,
%\\ A_{23}(t_k,z)=A_{13}(t_k,z)-2t_k,
% && \quad A_{24}(t_k,z)=0.
%\end{eqnarray*}
\begin{eqnarray*}
A_{31}(t_k,z)=\frac{16z^2t_kA_{11}(t_k,z)}{1-t_k-4zt_k-16z^2t_k^2},
%\\ &=& \frac{8z^2t_k}{1-t_k-4zt_k-16z^2t_k^2}* \frac{16t_k z^2(-1+24t_k^2z^2+4t_k+6t_kz)}{(-1+2t_kz+8t_kz^2)(-1+16t_k^2z^2+2t_k+8t_kz)}
&& \quad A_{32}(t_{k-1}, t_k,z)= \frac{16z^2t_{k-1}^{-1}t_k}{1-t_k-4zt_k-16z^2t_k^2},
\\ A_{33}(t_k,z)=  \frac{16z^2t_k+16z^2t_kA_{13}(t_k,z)}{1-t_k-4zt_k-16z^2t_k^2},
%\\ &=& -\frac{32t_k^2z^2(-1+t_k-z+2t_kz-4t_kz^2+8t_k^2z^2)}{(1-t_k-4zt_k-16z^2t_k^2)(1-16t_k^2z^2-2t_k-4t_kz)}
&& \quad A_{34}(t_{k-1}, t_k,z)=\frac{(1+4z)t_{k-1}^{-1}t_k+16z^2t_{k-1}^{-1}t_k^2}{1-t_k-4zt_k-16z^2t_k^2},
\\ A_{41}(t_k,z)=A_{31}(t_k,z),
&& \quad A_{42}(t_{k-1}, t_k,z)=A_{32}(t_{k-1}, t_k,z)-16z^2t_{k-1}^{-1}t_k,
\\ A_{43}(t_k,z)=A_{33}(t_k,z)-16z^2t_k,
&& \quad A_{44} (t_{k-1}, t_k,z)=A_{34}(t_{k-1}, t_k,z)-(1+4z)t_{k-1}^{-1}t_k.
\end{eqnarray*}
\end{lemma}
\begin{proof}
With a help of  Lemma \ref{pz-2} and Lemma  \ref{pz-1}, one  sees
\begin{eqnarray}
\label{50-1}
\begin{split}
\E^2(t_1,t_2,\cdots,t_k,z)&=16z^2t_k^2 t_{k-1}^{-1}\E^2(t_1,t_2,\cdots,t_{k-1},z)
\\&\quad +(t_k+4zt_k+16z^2t_k^2)\E^1(t_1,t_2,\cdots,t_k,z)+16z^2t_k\lambda^1(t_1,t_2,\cdots,t_k,z).
\end{split}
\end{eqnarray}
Using Lemmas \ref{pz-3}, \ref{pz-4}, we obtain
\begin{eqnarray}
\label{50-2}
\begin{split}
\E^2(t_1,t_2,\cdots,t_k,z)&=-(1+4z)t_{k-1}^{-1}t_k\E^2(t_1,t_2,\cdots,t_{k-1},z)+\E^1(t_1,t_2,\cdots,t_k,z)
\\ &-16z^2t_k\E^1(t_1,t_2,\cdots,t_{k-1},z)-16z^2t_{k-1}^{-1}t_k\lambda^2(t_1,t_2,\cdots,t_{k-1},z).
\end{split}
\end{eqnarray}
Combining (\ref{50-1}) with (\ref{5-3}) (\ref{50-2}),  one arrives at that
\begin{eqnarray*}
&& \E^1(t_1,t_2,\cdots,t_k,z)
\\ &&=\frac{1}{1-t_k-4zt_k-16z^2t_k^2}\Big\{
\big[
(1+4z)t_{k-1}^{-1}t_k+16z^2t_{k-1}^{-1}t_k^2\big]\E^2(t_1,t_2,\cdots,t_{k-1},z)
\\ &&\quad \quad+16z^2t_k\lambda^1(t_1,t_2,\cdots,t_k,z)
+16z^2t_k\E^1(t_1,t_2,\cdots,t_{k-1},z)
+16z^2t_{k-1}^{-1}t_k\lambda^2(t_1,t_2,\cdots,t_{k-1},z)\Big\},
\end{eqnarray*}
which finishes the proof of (\ref{6-1}).
Combining   (\ref{6-1}) with (\ref{50-2}), one arrives at (\ref{6-2}).
\end{proof}

Now, we state our  main results  in this section
\begin{theorem}
\label{a-1}
One has
\begin{eqnarray*}
\begin{bmatrix}
\lambda^1(t_1,t_2,\cdots,t_k,z)
\\
\lambda^2(t_1,t_2,\cdots,t_k,z)
\\
\E^1(t_1,t_2,\cdots,t_k,z)
\\
\E^2(t_1,t_2,\cdots,t_k,z)
\end{bmatrix}
=A(t_{k-1},t_k,z)
\begin{bmatrix}
\lambda^1(t_1,t_2,\cdots,t_{k-1},z)
\\
\lambda^2(t_1,t_2,\cdots,t_{k-1},z)
\\
\E^1(t_1,t_2,\cdots,t_{k-1},z)
\\
\E^2(t_1,t_2,\cdots,t_{k-1},z)
\end{bmatrix},
\end{eqnarray*}
where  $A(t_{k-1},t_k,z)$ is a matrix given by
$$
\begin{bmatrix}
A_{11}(t_k,z) &A_{12}(t_k,z)&A_{13}(t_k,z)&A_{14}(t_k,z)
\\ A_{21}(t_k,z)&A_{22}(t_k,z)&A_{23}(t_k,z)&A_{24}(t_k,z)
\\ A_{31}(t_k,z)&A_{32}(t_{k-1}, t_k,z)&A_{33}(t_k,z)&A_{34}(t_{k-1}, t_k,z)
\\ A_{41}(t_k,z)&A_{42}(t_{k-1}, t_k,z)&A_{43}(t_k,z)&A_{44}(t_{k-1}, t_k,z)
\end{bmatrix}.
$$
And furthermore,  we have
\begin{eqnarray*}
\begin{bmatrix}
\lambda^1(t_1,t_2,\cdots,t_k,z)
\\
\lambda^2(t_1,t_2,\cdots,t_k,z)
\\
\E^1(t_1,t_2,\cdots,t_k,z)
\\
\E^2(t_1,t_2,\cdots,t_k,z)
\end{bmatrix}
=A(t_{k-1},t_k,z)\cdot \cdots  \cdot A(t_{2},t_3,z)  \cdot
\begin{bmatrix}
\lambda^1(t_1,t_{2},z)
\\
\lambda^2(t_1,t_{2},z)
\\
\E^1(t_1,t_{2},z)
\\
\E^2(t_1,,t_{2},z)
\end{bmatrix},
\end{eqnarray*}
where the values of $\lambda^i(t_1,t_{2},z),\E^i(t_1,t_{2},z)$ are given in subsection 4.2.
\end{theorem}

For any $k$, by Theorem  \ref{a-1},  one  can derive the expression of  $\E^1(t_1,t_2,\cdots,t_k,z).$
Since
\begin{eqnarray*}
  \E_{m_1,\cdots,m_k}(z)
  =\frac{1}{m_1!\cdots m_k!}\left.\frac{\partial^{m_1+\cdots+m_k}\E^1(t_1,t_2,\cdots,t_k,z) }
  {\partial_{t_1}^{m_1}\partial_{t_2}^{m_2}\cdots\partial_{t_k}^{m_k}}\right|_{t_1=0,\cdots,t_k=0},
\end{eqnarray*}
 we can obtain  the expression of $\E_{{m_1,\cdots,m_k}}(z)$ for any fixed $ m_1\geq1,\cdots,m_k\geq 1$.

\section{Examples}
%In subsection  \ref{8-1}, we present some examples.
%\subsection{Examples}
\label{8-1}
For any $m_1\geq 1,\cdots,m_k\geq 1,$ by the results in Section 4,  one  gets
  the expression of $\E_{{m_1,\cdots,m_k}}(z)$. We give  two examples  to demonstrate this.
\begin{example}
\label{1-22}
  Let $k=2$.  We  have
  \begin{eqnarray*}
 \E_{{1,1}}(z)&=&2( 1 + 11 z + 80 z^2 + 212 z^3 + 208 z^4),
 \\
  \E_{{1,2}}(z)&=&2(1 + 15 z + 156 z^2 + 724 z^3 + 1728 z^4 + 1472 z^5),
   \\
   \E_{{2,2}}(z)&=&2(1 + 19 z + 248 z^2 + 1668 z^3 + 6704 z^4 + 13504 z^5 + 10624 z^6),
   \\
  \E_{{2,4}}(z)&=& 2(1 + 27 z + 544 z^2 + 5876 z^3 + 41424 z^4 + 185472 z^5 + 520320 z^6
  \\ &&  +
 813056 z^7 + 530432 z^8).
  \end{eqnarray*}
\end{example}

\begin{example}
  Let $k=3$. We   have
  \begin{eqnarray*}
 \E_{{1,1,1}}(z)&=&2( 1 + 19 z + 264 z^2 + 1748 z^3 + 6800 z^4 + 13440 z^5 + 10496 z^6),
 \\
  \E_{{1,2,3}}(z)&=&2(1 + 31 z + 700 z^2 + 9028 z^3 + 79840 z^4 + 465280 z^5 +
 1800832 z^6
 \\ &&+ 4425216 z^7+ 6236160 z^8 + 3760128 z^9),
 \\
 \E_{{2,2,2}}(z)&=&2( 1 + 31 z + 684 z^2 + 8756 z^3 + 75968 z^4 + 443456 z^5 +
 1750528 z^6
 \\ && + 4397056 z^7+ 6287360 z^8 + 3813376 z^9).
  \end{eqnarray*}
\end{example}
%\subsection{Conclusion}

%A special case of Halin graph  is the  wheel graph $W_n$.
%  Chen et al. \cite{CGM18}  obtained the genus polynomial of  $W_n$.
%  \begin{problem}
%  A natural problem is to
%   find the  recursive formulas or  explicit formulas for  $\cE_{W_n}(x),$
% where  $\cE_{W_n}(x)$ is the Euler-genus polynomial of  $W_n$.
%  \end{problem}

\noindent$\bf{Acknowledgement}$
We   greatly appreciate professor  Yichao  Chen's    useful discussion.

%For $m=19,$  Figure \ref{200-3} demonstrate  the changes of  histograms
%For the case $m=19,n=67790$,  using  the soft \emph{Mathematica}, we obtain
%$\delta=0.00345783$ and %$\delta=0.0109498$
%$
%  \sum_{i=1}^m |s_i-t_i|=0.0211478. %\sum_{i=1}^m |s_i-t_i|=0.0669679
%$

%===========================================================================


\begin{thebibliography}{WWW}
\bibitem{CGR94}
 J. Chen, J. L. Gross, and R. G. Rieper, Overlap matrices and total embeddings, \emph{Discrete Mathematics}  \textbf{128} (1994), 73-94.

\bibitem{CMZ16}
Y. Chen, T.Mansour, and Q. Zou, Embedding distributions and chebyshev polynomials, \emph {Graphs and Combinatorics} \textbf{28} (2012) 597-614.

\bibitem{COZ11}
 Y. Chen, L. Ou, and Q. Zou, Total embedding distributions of Ringel ladders, \emph{Discrete Mathematics}. \textbf{311} (2011) 2463-2474.

%\bibitem{CGM13}Y. Chen, J.L. Gross, and T. Mansour,  Total Embedding Distributions of Circular Ladders,  \emph{Journal of Graph Theory}. \textbf{74} (2013),32-57.
\bibitem{CG18} Y. Chen and J.L. Gross, An Euler-genus approach to the calculation of crosscap-number polynomial, \emph{Journal of Graph Theory}  \textbf{88} (2018) 80--100.


\bibitem{CGM18}  Y. Chen, J.L. Gross, and T. Mansour, {On the genus distributions of wheels and of related graphs}, \textsl{Discrete Mathematics} \textbf{341} (2018) 934--945.





%\bibitem{GT87} J. L. Gross and T. W. Tucker,
%	\textsl{Topological Graph Theory}, Dover, 2001; (original edn. Wiley, 1987).



%\bibitem{Gro11a} J. L. Gross,
%	Genus distributions of cubic outerplanar graphs,
%	 \textsl{J. Graph Algorithms and Applications} \textbf{15} (2011), 295--316.



\bibitem{Gro11}
J. L. Gross, Embeddings of cubic Halin graphs: Genus distributions, \emph{Ars Mathematica Contemporanea}. \textbf{6} (2011),37-56.

%\bibitem{Gro14}J. L. Gross, Embeddings of graphs of fixed treewidth and bounded degree,
%\textsl{Ars Mathematica Contemporanea} \textbf{7} (2014), 379-403


\bibitem{Ena18} E. Enami, Inequivalent embeddings of 3-connected 3-regular planar graphs on the torus, arXiv:1806.11333, 2018.

%    \bibitem{KS02}
%J.H. Kwak,  S.H. Shim,   Total embedding distributions for bouquets of circles,
% \emph{Discrete Mathematics} \textbf{248}(1)(2002), 93-108.

 \bibitem{Moh89}
 B. Mohar, An obstruction to embedding graphs in surfaces, \emph{Discrete Mathematics}. \textbf{78} (1989),
135-142.
\bibitem{Moh15} B. Mohar,
	The genus distribution of doubly hexagonal chains,
	``Topics in Chemical Graph Theory", I. Gutman, Ed.,
	\textsl{Mathematical Chemistry Monographs} Vol. 16a, Univ. Kragujevac, Kragujevac, (2014) 205--214.

\bibitem{PKG11} M. I.  Posgni,    I. F.  Khan, and  J.L.  Gross,   Genus distributions of 4-regular outerplanar graphs. \emph{Electronic Journal of Combinatorics}, \textbf{18} (2011), 1342-1353.
%
%\bibitem{Rin74}G.Ringel,Map Color Theory,Springer,Berlin,1974.



\bibitem{Sta91}S. Stahl, Permutation-partition pairs III: Embedding distributions of linear families of graphs, \emph{Journal of Combinatorial Theory, Series B} \textbf{52 }(1991), 191-218.





%\bibitem{WL08} L. X. Wan and Y. P. Liu, Orientable embedding genus distribution for certain types of graphs,\emph{ J. Combin. Theory (B).} \textbf{47} (2008), 19-32.\uppercase\expandafter{\romannumeral2}.



%
%
%\bibitem{Jin18}J. Zhang and X.H.  Peng, Genus distributions for cubic caterpillar-Halin graphs, Preprint, 2018.
\end{thebibliography}
\end{document}